\documentclass{article}

\input{hex-path.sty}

\title{Longest winning paths in Hex}
\author{Peter Selinger}
\date{Dalhousie University}

\noshadows

\begin{document}

\maketitle

\begin{abstract}
  We answer the question: what is the longest winning path on a Hex
  board of size $n\times n$?
\end{abstract}

% ----------------------------------------------------------------------
\section{Winning paths in Hex}

The game of Hex is played between two players on a rhombic board made
of hexagonal cells, like this:
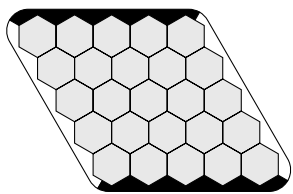
\[
\begin{hexboard}[scale=0.7]
  \rotation{-30}
  \board(5,5)
\end{hexboard}
\]
The players, Black and White, take turns putting a stone of their
color on some cell of the board. Stones are never moved or
removed. Black's goal is to connect the top edge to the bottom edge
with black stones, and White's goal is to connect the left edge to
the right edge with white stones.  Hex has many interesting
properties, among which is the fact that there is always exactly one
winner: the game cannot end in a draw {\cite{Gardner}}.

In this paper, we are interested in the length of possible winning
paths, answering a question of Tahir Yusufaly {\cite{Yusufaly}}.
Without loss of generality, we take Black's point of view. By a
\emph{winning connection}, we mean a set of black stones connecting
the top edge to the bottom edge. By a \emph{winning path}, we mean a
winning connection that is minimal, in the sense that none of its
proper subsets is a winning connection. The \emph{length} of a winning
path is the number of stones in it, and a winning path is
\emph{optimal} for a given board size if it is as long as
possible. Note that the question we are interested in is the existence
of winning paths, not whether such paths could occur in an actual
game.

Obviously, on a board of size $5\times 5$, Black needs at least 5
stones to make a winning path, since Black needs at least one stone in
each row, as shown in Figure~\ref{fig:ex-path}(a). On the other hand,
the longest possible winning path has length 11. There are 23
different winning paths of length 11, and a few of them are shown in
Figure~\ref{fig:ex-path}(b). The set of stones shown in
Figure~\ref{fig:ex-path}(c) is a winning connection, but not a winning
path, because it is not minimal: Black can remove 8 of the
stones and still have a winning connection.

% ......................................................................
\begin{figure}
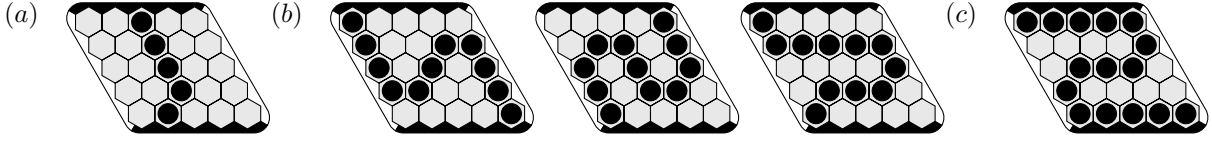

  \[
  (a)\quad
  \begin{hexboard}[scale=0.5, baseline={(0,0)}]
    \rotation{-30}
    \board(5,5)
    \black(3,1)
    \black(3,2)
    \black(3,3)
    \black(3,4)
    \black(2,5)
  \end{hexboard}
  (b)\quad
  \begin{hexboard}[scale=0.5, baseline={(0,0)}]
    \rotation{-30}
    \board(5,5)
    \black(1,1)
    \black(1,2)
    \black(1,3)
    \black(1,4)
    \black(2,4)
    \black(3,3)
    \black(4,2)
    \black(5,2)
    \black(5,3)
    \black(5,4)
    \black(5,5)
  \end{hexboard}
  \begin{hexboard}[scale=0.5, baseline={(0,0)}]
    \rotation{-30}
    \board(5,5)
    \black(5,1)
    \black(5,2)
    \black(5,3)
    \black(4,4)
    \black(3,4)
    \black(3,3)
    \black(3,2)
    \black(2,2)
    \black(1,3)
    \black(1,4)
    \black(1,5)
  \end{hexboard}
  \begin{hexboard}[scale=0.5, baseline={(0,0)}]
    \rotation{-30}
    \board(5,5)
    \black(1,1)
    \black(1,2)
    \black(2,2)
    \black(3,2)
    \black(4,2)
    \black(5,2)
    \black(5,3)
    \black(4,4)
    \black(3,4)
    \black(2,4)
    \black(1,5)
  \end{hexboard}
  (c)\quad
  \begin{hexboard}[scale=0.5, baseline={(0,0)}]
    \rotation{-30}
    \board(5,5)
    \black(1,1)
    \black(2,1)
    \black(3,1)
    \black(4,1)
    \black(5,1)
    \black(5,2)
    \black(4,3)
    \black(3,3)
    \black(2,3)
    \black(1,4)
    \black(1,5)
    \black(2,5)
    \black(3,5)
    \black(4,5)
    \black(5,5)
  \end{hexboard}
  \]  
  \caption{(a) A shortest winning path. (b) Some longest winning
    paths. (c) A non-minimal connection.}
  \label{fig:ex-path}
\end{figure}
% ......................................................................

For readers who like a good puzzle, consider the winning path for
a $10\times 10$ board that is shown in Figure~\ref{fig:tessel}(a).
This path has length 46 and is not optimal. Before reading on, try to
find a winning path of length 47.

We note that the problem of finding a longest winning path in Hex is
an instance of the {\em longest induced path problem}, which can be
asked of any graph. An induced path in a graph is a subset of the
vertices such that the induced subgraph is a path. Here, we are
interested in the specific version of this problem where the endpoints
of the path are fixed. In the case of a Hex board, the vertices are
the cells, as well as Black's two board edges. The edges are given by
adjacency. The longest induced path problem for a general graph is
NP-hard {\cite[p.~196]{GR1979}}; this remains true when the endpoints
are fixed. Nevertheless, finding good heuristic methods for solving
this problem remains a subject of current research; see
{\cite{MVPP2019}} for a recent contribution and additional references
and applications. It is therefore perhaps interesting that we can
solve this problem exactly in the case of a Hex grid.

% ----------------------------------------------------------------------
\section{Upper bounds on the path length}

% ----------------------------------------------------------------------
\subsection{Loose bounds}

Trivially, the length of a winning path on a board of size $n\times n$
is at most $n^2$. Less trivially, the length of a winning path is
bounded by $\frac{n^2+1}{2}$. Intuitively, it makes sense that only
``about half'' of the cells can be part of a winning path. A more
precise way of seeing this is as follows. Consider a winning path,
such as the one shown in Figure~\ref{fig:tessel}(a).  Add two columns
of empty cells to represent the left and right board edges, as in
Figure~\ref{fig:tessel}(b). Then create a triangular grid by
connecting the centers of cells, as in
Figure~\ref{fig:tessel}(c). This triangular grid consists of
$2(n+1)(n-1)$ triangles. We call it the \emph{unit grid} for an
$n\times n$-board, and we call the triangles the \emph{unit
  triangles}. Moreover we define the \emph{unit area} to be the area
of a unit triangle. Note that the stones sit on the vertices of the
unit grid. Next, we place each black stone inside a polygonal region
whose shape is determined by the position of the neighboring stones:
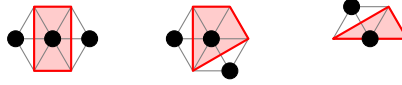
\[
\begin{hexboard}[scale=0.7]
  \def\smallradius{0.15}
  \rotation{-30}
  \begin{pgfonlayer}{hexes}
    \fill[domainpink] \coord(0,-1) -- \coord(-1,1) -- \coord(0,1) --
    \coord(1,-1) -- cycle;
    \draw[gridgray] \coord(0,-1) -- \coord(0,1);
    \draw[gridgray] \coord(1,-1) -- \coord(-1,1);
    \draw[gridgray] \coord(1,0) -- \coord(-1,0);
    \draw[gridgray] \coord(0,-1) -- \coord(1,-1) -- \coord(1,0) --
    \coord(0,1) -- \coord(-1,1) -- \coord(-1,0) -- cycle;
    \draw[domainred] \coord(0,-1) -- \coord(-1,1) -- \coord(0,1) --
    \coord(1,-1) -- cycle;
  \end{pgfonlayer}
  \smallblack(-1,0)
  \smallblack(0,0)
  \smallblack(1,0)

  \begin{scope}[shift={(3,0)}]
    \rotation{-30} % Need this to make sure \coord works correctly.
    \begin{pgfonlayer}{hexes}
      \fill[domainpink] \coord(0,-1) -- \coord(-1,1) -- \coord(1,0) --
      \coord(1,-1) -- cycle;
      \draw[gridgray] \coord(0,-1) -- \coord(0,1);
      \draw[gridgray] \coord(1,-1) -- \coord(-1,1);
      \draw[gridgray] \coord(1,0) -- \coord(-1,0);
      \draw[gridgray] \coord(0,-1) -- \coord(1,-1) -- \coord(1,0) --
      \coord(0,1) -- \coord(-1,1) -- \coord(-1,0) -- cycle;
      \draw[domainred] \coord(0,-1) -- \coord(-1,1) -- \coord(1,0) --
      \coord(1,-1) -- cycle;
    \end{pgfonlayer}
    \smallblack(-1,0)
    \smallblack(0,0)
    \smallblack(0,1)
  \end{scope}

  \begin{scope}[shift={(6,0)}]
    \rotation{-30} % Need this to make sure \coord works correctly.
    \begin{pgfonlayer}{hexes}
      \fill[domainpink] \coord(1,-1) -- \coord(-1,0) -- \coord(1,0) --
      \coord(1,-1) -- cycle;
      \draw[gridgray] \coord(0,-1) -- \coord(0,0);
      \draw[gridgray] \coord(1,-1) -- \coord(0,0);
      \draw[gridgray] \coord(1,0) -- \coord(-1,0);
      \draw[gridgray] \coord(0,-1) -- \coord(1,-1) -- \coord(1,0) --
      \coord(-1,0) -- cycle;
      \draw[domainred] \coord(1,-1) -- \coord(-1,0) -- \coord(1,0) --
      \coord(1,-1) -- cycle;
    \end{pgfonlayer}
    \smallblack(0,-1)
    \smallblack(0,0)
  \end{scope}
\end{hexboard}
\]
We call the region associated to each stone its \emph{domain}.  As
shown in Figure~\ref{fig:tessel}(d), the domains of all the stones in
a winning path are disjoint. Moreover, each domain occupies an area of
4 units, except for the domains of the boundary stones, which only
occupy an area of 2 units. So if there are $k$ stones in the path, we
have $4(k-1) \leq 2(n+1)(n-1)$, or equivalently $k\leq
\frac{n^2+1}{2}$, as claimed.

% ......................................................................
\begin{figure}
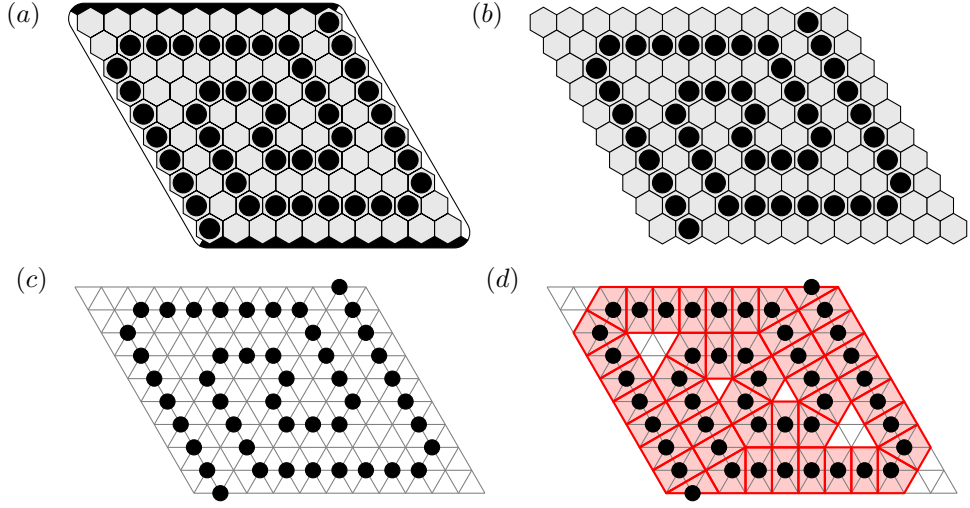

  \[
  (a)\quad
  \begin{hexboard}[scale=0.5,baseline={(0,0)}]
    \rotation{-30}
    \board(10,10)
    \black(10,1)\black(10,2)\black(10,3)\black(10,4)\black(10,5)
    \black(10,6)\black(10,7)\black(10,8)\black(9,9)\black(8,9)
    \black(7,9)\black(6,9)\black(5,9)\black(4,9)\black(3,9)
    \black(3,8)\black(3,7)\black(3,6)\black(3,5)\black(4,4)
    \black(5,4)\black(6,4)\black(6,5)\black(5,6)\black(5,7)
    \black(6,7)\black(7,7)\black(8,6)\black(8,5)\black(8,4)
    \black(8,3)\black(8,2)\black(7,2)\black(6,2)\black(5,2)
    \black(4,2)\black(3,2)\black(2,2)\black(1,3)\black(1,4)
    \black(1,5)\black(1,6)\black(1,7)\black(1,8)\black(1,9)
    \black(1,10)
  \end{hexboard}
  (b)\quad
  \begin{hexboard}[scale=0.5,baseline={(0,0)}]
    \rotation{-30}
    \foreach\i in {1,...,10} {
      \foreach\j in {0,...,11} {
      \hex(\j,\i)
      }
    }
    \black(10,1)\black(10,2)\black(10,3)\black(10,4)\black(10,5)
    \black(10,6)\black(10,7)\black(10,8)\black(9,9)\black(8,9)
    \black(7,9)\black(6,9)\black(5,9)\black(4,9)\black(3,9)
    \black(3,8)\black(3,7)\black(3,6)\black(3,5)\black(4,4)
    \black(5,4)\black(6,4)\black(6,5)\black(5,6)\black(5,7)
    \black(6,7)\black(7,7)\black(8,6)\black(8,5)\black(8,4)
    \black(8,3)\black(8,2)\black(7,2)\black(6,2)\black(5,2)
    \black(4,2)\black(3,2)\black(2,2)\black(1,3)\black(1,4)
    \black(1,5)\black(1,6)\black(1,7)\black(1,8)\black(1,9)
    \black(1,10)
  \end{hexboard}
  \]
  \[
  (c)\quad
  \begin{hexboard}[scale=0.5,baseline={(0,0)}]
    \rotation{-30}
    \begin{pgfonlayer}{hexes}
      \foreach\i in {1,...,10} {
        \draw[gridgray] \coord(0,\i) -- \coord(11,\i);
      }
      \foreach\j in {0,...,11} {
        \draw[gridgray] \coord(\j,1) -- \coord(\j,10);
      }
      \foreach\j in {1,...,9} {
        \draw[gridgray] \coord(\j,1) -- \coord(0,\j+1);
        \draw[gridgray] \coord(11-\j,10) -- \coord(11,10-\j);
      }
      \draw[gridgray] \coord(1,10) -- \coord(10,1);
    \end{pgfonlayer}
    \smallblack(10,1)\smallblack(10,2)\smallblack(10,3)
    \smallblack(10,4)\smallblack(10,5)\smallblack(10,6)
    \smallblack(10,7)\smallblack(10,8)\smallblack(9,9)
    \smallblack(8,9)\smallblack(7,9)\smallblack(6,9)
    \smallblack(5,9)\smallblack(4,9)\smallblack(3,9)
    \smallblack(3,8)\smallblack(3,7)\smallblack(3,6)
    \smallblack(3,5)\smallblack(4,4)\smallblack(5,4)
    \smallblack(6,4)\smallblack(6,5)\smallblack(5,6)
    \smallblack(5,7)\smallblack(6,7)\smallblack(7,7)
    \smallblack(8,6)\smallblack(8,5)\smallblack(8,4)
    \smallblack(8,3)\smallblack(8,2)\smallblack(7,2)
    \smallblack(6,2)\smallblack(5,2)\smallblack(4,2)
    \smallblack(3,2)\smallblack(2,2)\smallblack(1,3)
    \smallblack(1,4)\smallblack(1,5)\smallblack(1,6)
    \smallblack(1,7)\smallblack(1,8)\smallblack(1,9)
    \smallblack(1,10)
  \end{hexboard}
  (d)\quad
  \begin{hexboard}[scale=0.5,baseline={(0,0)}]
    \rotation{-30}
    \begin{pgfonlayer}{hexes}
      \foreach\i in {1,...,10} {
        \draw[gridgray] \coord(0,\i) -- \coord(11,\i);
      }
      \foreach\j in {0,...,11} {
        \draw[gridgray] \coord(\j,1) -- \coord(\j,10);
      }
      \foreach\j in {1,...,9} {
        \draw[gridgray] \coord(\j,1) -- \coord(0,\j+1);
        \draw[gridgray] \coord(11-\j,10) -- \coord(11,10-\j);
      }
      \draw[gridgray] \coord(1,10) -- \coord(10,1);
    \end{pgfonlayer}
    \region(10,1)(0,1)(0,2)(2,1)(2,1)
    \region(10,2)(0,1)(0,2)(2,1)(2,0)
    \region(10,3)(0,1)(0,2)(2,1)(2,0)
    \region(10,4)(0,1)(0,2)(2,1)(2,0)
    \region(10,5)(0,1)(0,2)(2,1)(2,0)
    \region(10,6)(0,1)(0,2)(2,1)(2,0)
    \region(10,7)(0,1)(0,2)(2,1)(2,0)
    \region(10,8)(0,1)(1,2)(2,1)(2,0)
    \region(9,9)(1,0)(0,2)(1,2)(2,1)
    \region(8,9)(1,0)(0,2)(1,2)(2,0)
    \region(7,9)(1,0)(0,2)(1,2)(2,0)
    \region(6,9)(1,0)(0,2)(1,2)(2,0)
    \region(5,9)(1,0)(0,2)(1,2)(2,0)
    \region(4,9)(1,0)(0,2)(1,2)(2,0)
    \region(3,9)(0,1)(0,2)(1,2)(2,0)
    \region(3,8)(0,1)(0,2)(2,1)(2,0)
    \region(3,7)(0,1)(0,2)(2,1)(2,0)
    \region(3,6)(0,1)(0,2)(2,1)(2,0)
    \region(3,5)(0,1)(0,2)(2,1)(1,0)
    \region(4,4)(0,1)(1,2)(2,0)(1,0)
    \region(5,4)(1,0)(0,2)(1,2)(2,0)
    \region(6,4)(1,0)(0,2)(2,1)(2,0)
    \region(6,5)(0,1)(1,2)(2,1)(2,0)
    \region(5,6)(0,1)(0,2)(2,1)(1,0)
    \region(5,7)(0,1)(0,2)(1,2)(2,0)
    \region(6,7)(1,0)(0,2)(1,2)(2,0)
    \region(7,7)(1,0)(0,2)(1,2)(2,1)
    \region(8,6)(0,1)(1,2)(2,1)(2,0)
    \region(8,5)(0,1)(0,2)(2,1)(2,0)
    \region(8,4)(0,1)(0,2)(2,1)(2,0)
    \region(8,3)(0,1)(0,2)(2,1)(2,0)
    \region(8,2)(1,0)(0,2)(2,1)(2,0)
    \region(7,2)(1,0)(0,2)(1,2)(2,0)
    \region(6,2)(1,0)(0,2)(1,2)(2,0)
    \region(5,2)(1,0)(0,2)(1,2)(2,0)
    \region(4,2)(1,0)(0,2)(1,2)(2,0)
    \region(3,2)(1,0)(0,2)(1,2)(2,0)
    \region(2,2)(0,1)(1,2)(2,0)(1,0)
    \region(1,3)(0,1)(0,2)(2,1)(1,0)
    \region(1,4)(0,1)(0,2)(2,1)(2,0)
    \region(1,5)(0,1)(0,2)(2,1)(2,0)
    \region(1,6)(0,1)(0,2)(2,1)(2,0)
    \region(1,7)(0,1)(0,2)(2,1)(2,0)
    \region(1,8)(0,1)(0,2)(2,1)(2,0)
    \region(1,9)(0,1)(0,2)(2,1)(2,0)
    \region(1,10)(0,1)(0,1)(2,1)(2,0)
  \end{hexboard}
  \]
  \caption{(a) A winning path for $10\times 10$. Can you find a longer
    one? (b) The same path, with the white edges replaced by columns
    of empty cells. (c) The triangular unit grid. (d) Each interior
    black stone occupies an area of 4 units.}
  \label{fig:tessel}
\end{figure}
% ......................................................................

% ----------------------------------------------------------------------
\subsection{Wasted triangles}

Later in this section, we will derive better bounds on the path
length. To motivate how this will be done, first consider the unit
triangles in Figure~\ref{fig:tessel}(d) that are not covered by the
domain of any stone. We call them \emph{wasted triangles}. There are
exactly 18 wasted triangles in Figure~\ref{fig:tessel}(d). We also
note that the same wasted triangles can equivalently be observed in
Figure~\ref{fig:tessel}(b), where they correspond to triples of
pairwise adjacent empty cells. For a given winning path, let $n$ be the
board size, let $k$ be the length of the path, and let $t$ be the
number of wasted triangles. These three quantities are related by a
simple formula. Namely, by expressing the area of the unit grid in two
different ways, we get $4(k-1) + t = 2(n+1)(n-1)$, or equivalently,
\begin{equation}\label{eqn:k}
  k = \frac{n^2+1}{2} - \frac{t}{4}.
\end{equation}
It follows that for any fixed board size, maximizing the path length
is equivalent to minimizing the number of wasted triangles. Our
strategy for finding an upper bound on the path length will be to find
a lower bound on the number of wasted triangles.

% ----------------------------------------------------------------------
\subsection{Wasted triangles near the corner}

It will be helpful to distinguish \emph{upward-pointing} unit
triangles ($\bigtriangleup$) from \emph{downward-pointing} ones
($\bigtriangledown$). The following lemma guarantees that there is at
least one upward-pointing wasted triangle near the top left corner of
the board.

% ......................................................................
\begin{figure}
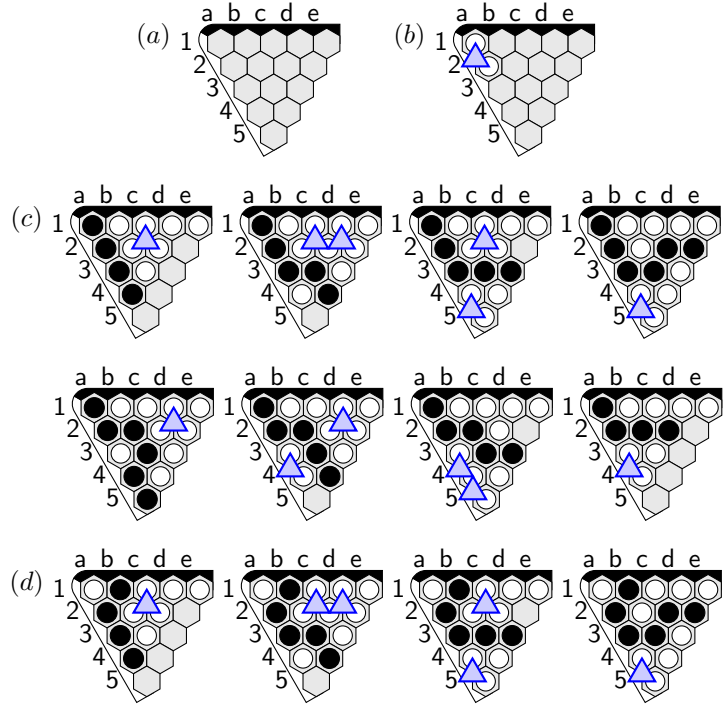

  \[
  (a)
  \begin{hexboard}[scale=0.5,baseline={(0,0)}]
    \rotation{-30}
    \foreach\i in {1,...,5} {\hex(\i,1)}
    \foreach\i in {1,...,4} {\hex(\i,2)}
    \foreach\i in {1,...,3} {\hex(\i,3)}
    \foreach\i in {1,...,2} {\hex(\i,4)}
    \foreach\i in {1,...,1} {\hex(\i,5)}
    \edge[\noobtusecorner](1,1)(1,5)
    \edge[\noobtusecorner](1,1)(5,1)
    \begin{scope}
      \tikzset{every node/.style={anchor=base,yshift=-0.6ex}}
      \cell(1.15,-0.3)\label{\lbl{a}}
      \cell(2.15,-0.3)\label{\lbl{b}}
      \cell(3.15,-0.3)\label{\lbl{c}}
      \cell(4.15,-0.3)\label{\lbl{d}}
      \cell(5.15,-0.3)\label{\lbl{e}}
      \cell(-0.3,1)\label{\lbl{1}}
      \cell(-0.3,2)\label{\lbl{2}}
      \cell(-0.3,3)\label{\lbl{3}}
      \cell(-0.3,4)\label{\lbl{4}}
      \cell(-0.3,5)\label{\lbl{5}}
    \end{scope}
  \end{hexboard}
  \qquad
  (b)
  \begin{hexboard}[scale=0.5,baseline={(0,0)}]
    \rotation{-30}
    \foreach\i in {1,...,5} {\hex(\i,1)}
    \foreach\i in {1,...,4} {\hex(\i,2)}
    \foreach\i in {1,...,3} {\hex(\i,3)}
    \foreach\i in {1,...,2} {\hex(\i,4)}
    \foreach\i in {1,...,1} {\hex(\i,5)}
    \edge[\noobtusecorner](1,1)(1,5)
    \edge[\noobtusecorner](1,1)(5,1)
    \begin{scope}
      \tikzset{every node/.style={anchor=base,yshift=-0.6ex}}
      \cell(1.15,-0.3)\label{\lbl{a}}
      \cell(2.15,-0.3)\label{\lbl{b}}
      \cell(3.15,-0.3)\label{\lbl{c}}
      \cell(4.15,-0.3)\label{\lbl{d}}
      \cell(5.15,-0.3)\label{\lbl{e}}
      \cell(-0.3,1)\label{\lbl{1}}
      \cell(-0.3,2)\label{\lbl{2}}
      \cell(-0.3,3)\label{\lbl{3}}
      \cell(-0.3,4)\label{\lbl{4}}
      \cell(-0.3,5)\label{\lbl{5}}
    \end{scope}
    \white(1,1)
    \white(1,2)
    \draw[wasted] \coord(1,2) -- \coord(1,1) -- \coord(0,2) -- cycle;
  \end{hexboard}
  \]
  \[
  (c)
  \begin{hexboard}[scale=0.5,baseline={(0,0)}]
    \rotation{-30}
    \foreach\i in {1,...,5} {\hex(\i,1)}
    \foreach\i in {1,...,4} {\hex(\i,2)}
    \foreach\i in {1,...,3} {\hex(\i,3)}
    \foreach\i in {1,...,2} {\hex(\i,4)}
    \foreach\i in {1,...,1} {\hex(\i,5)}
    \edge[\noobtusecorner](1,1)(1,5)
    \edge[\noobtusecorner](1,1)(5,1)
    \begin{scope}
      \tikzset{every node/.style={anchor=base,yshift=-0.6ex}}
      \cell(1.15,-0.3)\label{\lbl{a}}
      \cell(2.15,-0.3)\label{\lbl{b}}
      \cell(3.15,-0.3)\label{\lbl{c}}
      \cell(4.15,-0.3)\label{\lbl{d}}
      \cell(5.15,-0.3)\label{\lbl{e}}
      \cell(-0.3,1)\label{\lbl{1}}
      \cell(-0.3,2)\label{\lbl{2}}
      \cell(-0.3,3)\label{\lbl{3}}
      \cell(-0.3,4)\label{\lbl{4}}
      \cell(-0.3,5)\label{\lbl{5}}
    \end{scope}
    \black(1,1)
    \black(1,2)
    \black(1,3)
    \black(1,4)
    \white(2,1)
    \white(2,2)
    \white(2,3)
    \white(3,1)
    \white(3,2)
    \white(4,1)
    \white(5,1)
    \draw[wasted] \coord(3,2) -- \coord(3,1) -- \coord(2,2) -- cycle;
  \end{hexboard}
  \begin{hexboard}[scale=0.5,baseline={(0,0)}]
    \rotation{-30}
    \foreach\i in {1,...,5} {\hex(\i,1)}
    \foreach\i in {1,...,4} {\hex(\i,2)}
    \foreach\i in {1,...,3} {\hex(\i,3)}
    \foreach\i in {1,...,2} {\hex(\i,4)}
    \foreach\i in {1,...,1} {\hex(\i,5)}
    \edge[\noobtusecorner](1,1)(1,5)
    \edge[\noobtusecorner](1,1)(5,1)
    \begin{scope}
      \tikzset{every node/.style={anchor=base,yshift=-0.6ex}}
      \cell(1.15,-0.3)\label{\lbl{a}}
      \cell(2.15,-0.3)\label{\lbl{b}}
      \cell(3.15,-0.3)\label{\lbl{c}}
      \cell(4.15,-0.3)\label{\lbl{d}}
      \cell(5.15,-0.3)\label{\lbl{e}}
      \cell(-0.3,1)\label{\lbl{1}}
      \cell(-0.3,2)\label{\lbl{2}}
      \cell(-0.3,3)\label{\lbl{3}}
      \cell(-0.3,4)\label{\lbl{4}}
      \cell(-0.3,5)\label{\lbl{5}}
    \end{scope}
    \black(1,1)
    \black(1,2)
    \black(1,3)
    \black(2,3)
    \black(2,4)
    \white(2,1)
    \white(2,2)
    \white(3,1)
    \white(3,2)
    \white(4,1)
    \white(5,1)
    \white(3,3)
    \white(1,4)
    \white(4,2)
    \draw[wasted] \coord(3,2) -- \coord(3,1) -- \coord(2,2) -- cycle;
    \draw[wasted] \coord(4,2) -- \coord(4,1) -- \coord(3,2) -- cycle;
  \end{hexboard}
  \begin{hexboard}[scale=0.5,baseline={(0,0)}]
    \rotation{-30}
    \foreach\i in {1,...,5} {\hex(\i,1)}
    \foreach\i in {1,...,4} {\hex(\i,2)}
    \foreach\i in {1,...,3} {\hex(\i,3)}
    \foreach\i in {1,...,2} {\hex(\i,4)}
    \foreach\i in {1,...,1} {\hex(\i,5)}
    \edge[\noobtusecorner](1,1)(1,5)
    \edge[\noobtusecorner](1,1)(5,1)
    \begin{scope}
      \tikzset{every node/.style={anchor=base,yshift=-0.6ex}}
      \cell(1.15,-0.3)\label{\lbl{a}}
      \cell(2.15,-0.3)\label{\lbl{b}}
      \cell(3.15,-0.3)\label{\lbl{c}}
      \cell(4.15,-0.3)\label{\lbl{d}}
      \cell(5.15,-0.3)\label{\lbl{e}}
      \cell(-0.3,1)\label{\lbl{1}}
      \cell(-0.3,2)\label{\lbl{2}}
      \cell(-0.3,3)\label{\lbl{3}}
      \cell(-0.3,4)\label{\lbl{4}}
      \cell(-0.3,5)\label{\lbl{5}}
    \end{scope}
    \black(1,1)
    \black(1,2)
    \black(1,3)
    \black(2,3)
    \black(3,3)
    \white(2,1)
    \white(2,2)
    \white(3,1)
    \white(3,2)
    \white(4,1)
    \white(1,4)
    \white(2,4)
    \white(1,5)
    \white(5,1)
    \draw[wasted] \coord(3,2) -- \coord(3,1) -- \coord(2,2) -- cycle;
    \draw[wasted] \coord(1,5) -- \coord(1,4) -- \coord(0,5) -- cycle;
  \end{hexboard}
  \begin{hexboard}[scale=0.5,baseline={(0,0)}]
    \rotation{-30}
    \foreach\i in {1,...,5} {\hex(\i,1)}
    \foreach\i in {1,...,4} {\hex(\i,2)}
    \foreach\i in {1,...,3} {\hex(\i,3)}
    \foreach\i in {1,...,2} {\hex(\i,4)}
    \foreach\i in {1,...,1} {\hex(\i,5)}
    \edge[\noobtusecorner](1,1)(1,5)
    \edge[\noobtusecorner](1,1)(5,1)
    \begin{scope}
      \tikzset{every node/.style={anchor=base,yshift=-0.6ex}}
      \cell(1.15,-0.3)\label{\lbl{a}}
      \cell(2.15,-0.3)\label{\lbl{b}}
      \cell(3.15,-0.3)\label{\lbl{c}}
      \cell(4.15,-0.3)\label{\lbl{d}}
      \cell(5.15,-0.3)\label{\lbl{e}}
      \cell(-0.3,1)\label{\lbl{1}}
      \cell(-0.3,2)\label{\lbl{2}}
      \cell(-0.3,3)\label{\lbl{3}}
      \cell(-0.3,4)\label{\lbl{4}}
      \cell(-0.3,5)\label{\lbl{5}}
    \end{scope}
    \black(1,1)
    \black(1,2)
    \black(1,3)
    \black(2,3)
    \black(3,2)
    \white(1,4)
    \white(2,4)
    \white(1,5)
    \white(2,1)
    \white(3,1)
    \white(4,1)
    \white(5,1)
    \white(2,2)
    \white(3,3)
    \black(4,2)
    \draw[wasted] \coord(1,5) -- \coord(1,4) -- \coord(0,5) -- cycle;
  \end{hexboard}
  \]
  \[
  \quad~
  \begin{hexboard}[scale=0.5,baseline={(0,0)}]
    \rotation{-30}
    \foreach\i in {1,...,5} {\hex(\i,1)}
    \foreach\i in {1,...,4} {\hex(\i,2)}
    \foreach\i in {1,...,3} {\hex(\i,3)}
    \foreach\i in {1,...,2} {\hex(\i,4)}
    \foreach\i in {1,...,1} {\hex(\i,5)}
    \edge[\noobtusecorner](1,1)(1,5)
    \edge[\noobtusecorner](1,1)(5,1)
    \begin{scope}
      \tikzset{every node/.style={anchor=base,yshift=-0.6ex}}
      \cell(1.15,-0.3)\label{\lbl{a}}
      \cell(2.15,-0.3)\label{\lbl{b}}
      \cell(3.15,-0.3)\label{\lbl{c}}
      \cell(4.15,-0.3)\label{\lbl{d}}
      \cell(5.15,-0.3)\label{\lbl{e}}
      \cell(-0.3,1)\label{\lbl{1}}
      \cell(-0.3,2)\label{\lbl{2}}
      \cell(-0.3,3)\label{\lbl{3}}
      \cell(-0.3,4)\label{\lbl{4}}
      \cell(-0.3,5)\label{\lbl{5}}
    \end{scope}
    \black(1,1)
    \black(1,2)
    \black(2,2)
    \black(2,3)
    \black(1,4)
    \white(2,1)
    \white(3,1)
    \white(3,2)
    \white(3,3)
    \white(4,1)
    \white(5,1)
    \white(4,2)
    \white(1,3)
    \white(2,4)
    \black(1,5)
    \draw[wasted] \coord(4,2) -- \coord(4,1) -- \coord(3,2) -- cycle;
  \end{hexboard}
  \begin{hexboard}[scale=0.5,baseline={(0,0)}]
    \rotation{-30}
    \foreach\i in {1,...,5} {\hex(\i,1)}
    \foreach\i in {1,...,4} {\hex(\i,2)}
    \foreach\i in {1,...,3} {\hex(\i,3)}
    \foreach\i in {1,...,2} {\hex(\i,4)}
    \foreach\i in {1,...,1} {\hex(\i,5)}
    \edge[\noobtusecorner](1,1)(1,5)
    \edge[\noobtusecorner](1,1)(5,1)
    \begin{scope}
      \tikzset{every node/.style={anchor=base,yshift=-0.6ex}}
      \cell(1.15,-0.3)\label{\lbl{a}}
      \cell(2.15,-0.3)\label{\lbl{b}}
      \cell(3.15,-0.3)\label{\lbl{c}}
      \cell(4.15,-0.3)\label{\lbl{d}}
      \cell(5.15,-0.3)\label{\lbl{e}}
      \cell(-0.3,1)\label{\lbl{1}}
      \cell(-0.3,2)\label{\lbl{2}}
      \cell(-0.3,3)\label{\lbl{3}}
      \cell(-0.3,4)\label{\lbl{4}}
      \cell(-0.3,5)\label{\lbl{5}}
    \end{scope}
    \black(1,1)
    \black(1,2)
    \black(2,2)
    \black(2,3)
    \black(2,4)
    \white(2,1)
    \white(3,1)
    \white(3,2)
    \white(3,3)
    \white(4,1)
    \white(5,1)
    \white(4,2)
    \white(1,3)
    \white(1,4)
    \draw[wasted] \coord(4,2) -- \coord(4,1) -- \coord(3,2) -- cycle;
    \draw[wasted] \coord(1,4) -- \coord(1,3) -- \coord(0,4) -- cycle;
  \end{hexboard}
  \begin{hexboard}[scale=0.5,baseline={(0,0)}]
    \rotation{-30}
    \foreach\i in {1,...,5} {\hex(\i,1)}
    \foreach\i in {1,...,4} {\hex(\i,2)}
    \foreach\i in {1,...,3} {\hex(\i,3)}
    \foreach\i in {1,...,2} {\hex(\i,4)}
    \foreach\i in {1,...,1} {\hex(\i,5)}
    \edge[\noobtusecorner](1,1)(1,5)
    \edge[\noobtusecorner](1,1)(5,1)
    \begin{scope}
      \tikzset{every node/.style={anchor=base,yshift=-0.6ex}}
      \cell(1.15,-0.3)\label{\lbl{a}}
      \cell(2.15,-0.3)\label{\lbl{b}}
      \cell(3.15,-0.3)\label{\lbl{c}}
      \cell(4.15,-0.3)\label{\lbl{d}}
      \cell(5.15,-0.3)\label{\lbl{e}}
      \cell(-0.3,1)\label{\lbl{1}}
      \cell(-0.3,2)\label{\lbl{2}}
      \cell(-0.3,3)\label{\lbl{3}}
      \cell(-0.3,4)\label{\lbl{4}}
      \cell(-0.3,5)\label{\lbl{5}}
    \end{scope}
    \black(1,1)
    \black(1,2)
    \black(2,2)
    \black(2,3)
    \black(3,3)
    \white(1,3)
    \white(1,4)
    \white(1,5)
    \white(2,4)
    \white(2,1)
    \white(3,1)
    \white(4,1)
    \white(5,1)
    \white(3,2)
    \draw[wasted] \coord(1,4) -- \coord(1,3) -- \coord(0,4) -- cycle;
    \draw[wasted] \coord(1,5) -- \coord(1,4) -- \coord(0,5) -- cycle;
  \end{hexboard}
  \begin{hexboard}[scale=0.5,baseline={(0,0)}]
    \rotation{-30}
    \foreach\i in {1,...,5} {\hex(\i,1)}
    \foreach\i in {1,...,4} {\hex(\i,2)}
    \foreach\i in {1,...,3} {\hex(\i,3)}
    \foreach\i in {1,...,2} {\hex(\i,4)}
    \foreach\i in {1,...,1} {\hex(\i,5)}
    \edge[\noobtusecorner](1,1)(1,5)
    \edge[\noobtusecorner](1,1)(5,1)
    \begin{scope}
      \tikzset{every node/.style={anchor=base,yshift=-0.6ex}}
      \cell(1.15,-0.3)\label{\lbl{a}}
      \cell(2.15,-0.3)\label{\lbl{b}}
      \cell(3.15,-0.3)\label{\lbl{c}}
      \cell(4.15,-0.3)\label{\lbl{d}}
      \cell(5.15,-0.3)\label{\lbl{e}}
      \cell(-0.3,1)\label{\lbl{1}}
      \cell(-0.3,2)\label{\lbl{2}}
      \cell(-0.3,3)\label{\lbl{3}}
      \cell(-0.3,4)\label{\lbl{4}}
      \cell(-0.3,5)\label{\lbl{5}}
    \end{scope}
    \black(1,1)
    \black(1,2)
    \black(2,2)
    \black(3,2)
    \white(1,3)
    \white(1,4)
    \white(2,3)
    \white(2,1)
    \white(3,1)
    \white(4,1)
    \white(5,1)
    \draw[wasted] \coord(1,4) -- \coord(1,3) -- \coord(0,4) -- cycle;
  \end{hexboard}
  \]
  \[
  (d)
  \begin{hexboard}[scale=0.5,baseline={(0,0)}]
    \rotation{-30}
    \foreach\i in {1,...,5} {\hex(\i,1)}
    \foreach\i in {1,...,4} {\hex(\i,2)}
    \foreach\i in {1,...,3} {\hex(\i,3)}
    \foreach\i in {1,...,2} {\hex(\i,4)}
    \foreach\i in {1,...,1} {\hex(\i,5)}
    \edge[\noobtusecorner](1,1)(1,5)
    \edge[\noobtusecorner](1,1)(5,1)
    \begin{scope}
      \tikzset{every node/.style={anchor=base,yshift=-0.6ex}}
      \cell(1.15,-0.3)\label{\lbl{a}}
      \cell(2.15,-0.3)\label{\lbl{b}}
      \cell(3.15,-0.3)\label{\lbl{c}}
      \cell(4.15,-0.3)\label{\lbl{d}}
      \cell(5.15,-0.3)\label{\lbl{e}}
      \cell(-0.3,1)\label{\lbl{1}}
      \cell(-0.3,2)\label{\lbl{2}}
      \cell(-0.3,3)\label{\lbl{3}}
      \cell(-0.3,4)\label{\lbl{4}}
      \cell(-0.3,5)\label{\lbl{5}}
    \end{scope}
    \black(2,1)
    \black(1,2)
    \black(1,3)
    \black(1,4)
    \white(1,1)
    \white(2,2)
    \white(2,3)
    \white(3,1)
    \white(3,2)
    \white(4,1)
    \white(5,1)
    \draw[wasted] \coord(3,2) -- \coord(3,1) -- \coord(2,2) -- cycle;
  \end{hexboard}
  \begin{hexboard}[scale=0.5,baseline={(0,0)}]
    \rotation{-30}
    \foreach\i in {1,...,5} {\hex(\i,1)}
    \foreach\i in {1,...,4} {\hex(\i,2)}
    \foreach\i in {1,...,3} {\hex(\i,3)}
    \foreach\i in {1,...,2} {\hex(\i,4)}
    \foreach\i in {1,...,1} {\hex(\i,5)}
    \edge[\noobtusecorner](1,1)(1,5)
    \edge[\noobtusecorner](1,1)(5,1)
    \begin{scope}
      \tikzset{every node/.style={anchor=base,yshift=-0.6ex}}
      \cell(1.15,-0.3)\label{\lbl{a}}
      \cell(2.15,-0.3)\label{\lbl{b}}
      \cell(3.15,-0.3)\label{\lbl{c}}
      \cell(4.15,-0.3)\label{\lbl{d}}
      \cell(5.15,-0.3)\label{\lbl{e}}
      \cell(-0.3,1)\label{\lbl{1}}
      \cell(-0.3,2)\label{\lbl{2}}
      \cell(-0.3,3)\label{\lbl{3}}
      \cell(-0.3,4)\label{\lbl{4}}
      \cell(-0.3,5)\label{\lbl{5}}
    \end{scope}
    \black(2,1)
    \black(1,2)
    \black(1,3)
    \black(2,3)
    \black(2,4)
    \white(1,1)
    \white(2,2)
    \white(3,1)
    \white(3,2)
    \white(4,1)
    \white(5,1)
    \white(3,3)
    \white(1,4)
    \white(4,2)
    \draw[wasted] \coord(3,2) -- \coord(3,1) -- \coord(2,2) -- cycle;
    \draw[wasted] \coord(4,2) -- \coord(4,1) -- \coord(3,2) -- cycle;
  \end{hexboard}
  \begin{hexboard}[scale=0.5,baseline={(0,0)}]
    \rotation{-30}
    \foreach\i in {1,...,5} {\hex(\i,1)}
    \foreach\i in {1,...,4} {\hex(\i,2)}
    \foreach\i in {1,...,3} {\hex(\i,3)}
    \foreach\i in {1,...,2} {\hex(\i,4)}
    \foreach\i in {1,...,1} {\hex(\i,5)}
    \edge[\noobtusecorner](1,1)(1,5)
    \edge[\noobtusecorner](1,1)(5,1)
    \begin{scope}
      \tikzset{every node/.style={anchor=base,yshift=-0.6ex}}
      \cell(1.15,-0.3)\label{\lbl{a}}
      \cell(2.15,-0.3)\label{\lbl{b}}
      \cell(3.15,-0.3)\label{\lbl{c}}
      \cell(4.15,-0.3)\label{\lbl{d}}
      \cell(5.15,-0.3)\label{\lbl{e}}
      \cell(-0.3,1)\label{\lbl{1}}
      \cell(-0.3,2)\label{\lbl{2}}
      \cell(-0.3,3)\label{\lbl{3}}
      \cell(-0.3,4)\label{\lbl{4}}
      \cell(-0.3,5)\label{\lbl{5}}
    \end{scope}
    \black(2,1)
    \black(1,2)
    \black(1,3)
    \black(2,3)
    \black(3,3)
    \white(1,1)
    \white(2,2)
    \white(3,1)
    \white(3,2)
    \white(4,1)
    \white(1,4)
    \white(2,4)
    \white(1,5)
    \white(5,1)
    \draw[wasted] \coord(3,2) -- \coord(3,1) -- \coord(2,2) -- cycle;
    \draw[wasted] \coord(1,5) -- \coord(1,4) -- \coord(0,5) -- cycle;
  \end{hexboard}
  \begin{hexboard}[scale=0.5,baseline={(0,0)}]
    \rotation{-30}
    \foreach\i in {1,...,5} {\hex(\i,1)}
    \foreach\i in {1,...,4} {\hex(\i,2)}
    \foreach\i in {1,...,3} {\hex(\i,3)}
    \foreach\i in {1,...,2} {\hex(\i,4)}
    \foreach\i in {1,...,1} {\hex(\i,5)}
    \edge[\noobtusecorner](1,1)(1,5)
    \edge[\noobtusecorner](1,1)(5,1)
    \begin{scope}
      \tikzset{every node/.style={anchor=base,yshift=-0.6ex}}
      \cell(1.15,-0.3)\label{\lbl{a}}
      \cell(2.15,-0.3)\label{\lbl{b}}
      \cell(3.15,-0.3)\label{\lbl{c}}
      \cell(4.15,-0.3)\label{\lbl{d}}
      \cell(5.15,-0.3)\label{\lbl{e}}
      \cell(-0.3,1)\label{\lbl{1}}
      \cell(-0.3,2)\label{\lbl{2}}
      \cell(-0.3,3)\label{\lbl{3}}
      \cell(-0.3,4)\label{\lbl{4}}
      \cell(-0.3,5)\label{\lbl{5}}
    \end{scope}
    \black(2,1)
    \black(1,2)
    \black(1,3)
    \black(2,3)
    \black(3,2)
    \white(1,4)
    \white(2,4)
    \white(1,5)
    \white(1,1)
    \white(3,1)
    \white(4,1)
    \white(5,1)
    \white(2,2)
    \white(3,3)
    \black(4,2)
    \draw[wasted] \coord(1,5) -- \coord(1,4) -- \coord(0,5) -- cycle;
  \end{hexboard}
  \]
  \caption{(a) The 15-cell corner region. (b) An upward-pointing
    wasted triangle when \lbl{a1} and \lbl{b1} are white. }
  \label{fig:wasted}
\end{figure}
% ......................................................................

\begin{lemma}\label{lem:upward}
  Consider a winning path on a board of size at least $5\times
  5$. Then there exists an upward-pointing wasted triangle within the
  15-cell corner region shown in Figure~\ref{fig:wasted}(a).
\end{lemma}

\begin{proof}
  By case distinction. Consider the board region in
  Figure~\ref{fig:wasted}(a). We use standard Hex coordinates; for
  example, the cell in the acute corner is \lbl{a1}, and its neighbors
  are \lbl{b1} and \lbl{a2}. For brevity, we refer to cells that are
  part of the winning path as ``black'' and all other cells as
  ``white''. Recall that a wasted triangle can overlap the left
  edge (cf.\@ Figure~\ref{fig:tessel} (b) and (d)).

  \begin{itemize}
  \item Case 1: \lbl{a2} is white. Then \lbl{a1} must also be white,
    since no minimal winning path can pass through it. Then \lbl{a1},
    \lbl{a2}, and the left edge form an upward-pointing wasted
    triangle, as shown in Figure~\ref{fig:wasted}(b).
  \item Case 2: \lbl{a2} and \lbl{a1} are black. We distinguish
    several subcases, based on how the path continues. These subcases
    are shown in Figure~\ref{fig:wasted}(c). In each case, we know a
    cell to be white if it is next to the top edge, if it already has two
    adjacent black neighbors, if its unique black neighbor already has
    two black neighbors, or if there is no room for it to have two
    non-adjacent black neighbors. In all cases, this is sufficient to
    identify at least one upward-pointing wasted triangle (and
    sometimes more than one).
  \item Case 3: \lbl{a2} is black and \lbl{a1} is white. In this case,
    \lbl{b1} must be black, and the situation is similar to the
    previous case, but slightly simpler. It is shown in
    Figure~\ref{fig:wasted}(d).\qedhere
  \end{itemize}
\end{proof}

% ----------------------------------------------------------------------
\subsection{Boundary components in triangular regions}

To motivate the next construction, we return to the unit grid of
Figure~\ref{fig:tessel}(d), with its winning path and the domains of
the path stones. Consider the following triangle-shaped subregion:
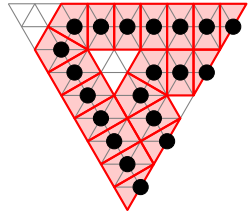
\[
\begin{hexboard}[scale=0.5,baseline={(0,0)}]
  \rotation{-30}
  \begin{pgfonlayer}{hexes}
    \foreach\i in {1,...,9} {
      \draw[gridgray] \coord(0,\i) -- \coord(10-\i,\i);
    }
    \foreach\j in {0,...,8} {
      \draw[gridgray] \coord(\j,1) -- \coord(\j,10-\j);
    }
    \foreach\j in {1,...,9} {
      \draw[gridgray] \coord(\j,1) -- \coord(0,\j+1);
    }
  \end{pgfonlayer}
  \region(3,7)(0,1)(0,2)(2,0)(2,0)
  \region(3,6)(0,1)(0,2)(2,1)(2,0)
  \region(3,5)(0,1)(0,2)(2,1)(1,0)
  \region(4,4)(0,1)(1,2)(2,0)(1,0)
  \region(5,4)(1,0)(0,2)(1,2)(2,0)
  \region(6,4)(1,0)(0,2)(2,0)(2,0)
  \region(8,2)(1,0)(0,2)(2,0)(2,0)
  \region(7,2)(1,0)(0,2)(1,2)(2,0)
  \region(6,2)(1,0)(0,2)(1,2)(2,0)
  \region(5,2)(1,0)(0,2)(1,2)(2,0)
  \region(4,2)(1,0)(0,2)(1,2)(2,0)
  \region(3,2)(1,0)(0,2)(1,2)(2,0)
  \region(2,2)(0,1)(1,2)(2,0)(1,0)
  \region(1,3)(0,1)(0,2)(2,1)(1,0)
  \region(1,4)(0,1)(0,2)(2,1)(2,0)
  \region(1,5)(0,1)(0,2)(2,1)(2,0)
  \region(1,6)(0,1)(0,2)(2,1)(2,0)
  \region(1,7)(0,1)(0,2)(2,1)(2,0)
  \region(1,8)(0,1)(0,2)(2,1)(2,0)
  \region(1,9)(0,1)(0,2)(2,0)(2,0)
\end{hexboard}
\]
This particular region of the unit grid consists of 45
downward-pointing unit triangles (wasted or otherwise) and 36
upward-pointing ones. Therefore, the total number of downward-pointing
unit triangles exceeds the number of upward-pointing ones by 9 in this
region. Note that the domain of each interior black stone partially or
fully overlaps some downward-pointing and some upward-pointing unit
triangles. By area, each such domain covers an equal amount of
downward- and upward-pointing unit triangles (2 units of each).  On
the other hand, the domain of each black boundary stone covers 1.5
downward-pointing and 0.5 upward-pointing unit triangles by area. It
follows that the total excess of 9 upward-pointing triangles in the
region can be accounted for as follows: 4 of them are due to boundary
stones, and 5 are due to an excess of downward-pointing wasted
triangles over upward-pointing ones. In other words, in the presence
of this region's 4 boundary stones, there must be 5 more
downward-pointing wasted triangles than upward-pointing ones.

This reasoning generalizes to other triangle-shaped regions. Consider
a winning path on some Hex board, and consider some downward-pointing
triangle-shaped region of the unit grid. We first examine what happens
along the region's boundary. Two stones of the winning path are
\emph{consecutive} if they are adjacent. We call a maximal stretch of
consecutive stones that lie on the region's boundary a \emph{boundary
component}. Figure~\ref{fig:boundary} shows some possible shapes of
boundary components for the triangular region. In fact,
Figure~\ref{fig:boundary} represents all possible shapes of boundary
components, up to symmetry, the direction of neighboring path segments
outside the region, and the number of repetitions of horizontal path
segments. We say that a boundary component is \emph{transversal} if it
is of the form shown in Figure~\ref{fig:boundary}$(b)$ or $(c)$.

% ......................................................................
\begin{figure}
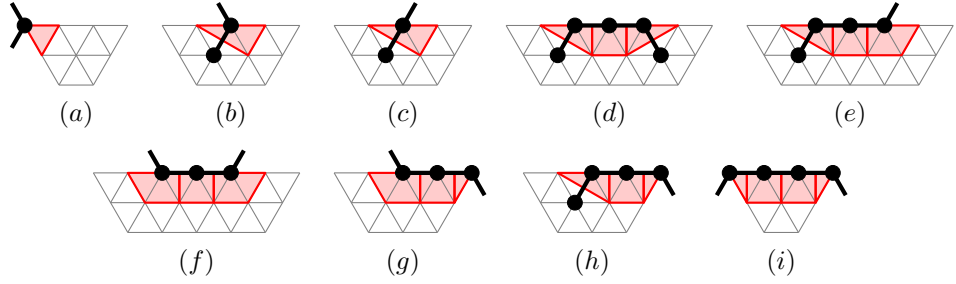

\[
\begin{hexboard}[scale=0.65,baseline={(0,0)}]
  \def\smallradius{0.15}
  \begin{scope}[xshift=0cm]
    \rotation{-30}
    \grid{0}{3}
    \region(-1,1)(1,1)(1,2)(2,1)(2,1)
    \draw[ultra thick] \coord(-1.75,1.75) -- \coord(-1,1) -- \coord(-1,0.25);
    \path
    \coord(-1,4) node {$(a)$}
    ;
  \end{scope}
  \begin{scope}[xshift=2.8cm]
    \rotation{-30}
    \grid{0}{4}
    \smallblack(0,2)
    \region(1,1)(0,1)(1,2)(2,1)(2,1)
    \draw[ultra thick] \coord(0,2) -- \coord(1,1) -- \coord(1,0.25);
    \path
    \coord(-0.5,4) node {$(b)$}
    ;
  \end{scope}
  \begin{scope}[xshift=6.3cm]
    \rotation{-30}
    \grid{0}{4}
    \smallblack(0,2)
    \region(1,1)(0,1)(1,2)(2,1)(2,1)
    \draw[ultra thick] \coord(0,2) -- \coord(1,1) -- \coord(1.75,0.25);
    \path
    \coord(-0.5,4) node {$(c)$}
    ;
  \end{scope}
  \begin{scope}[xshift=9.8cm]
    \rotation{-30}
    \grid{0}{6}
    \smallblack(0,2)
    \region(1,1)(0,1)(1,2)(1.5,1)(1.5,1)
    \region(2,1)(0.5,1)(0,2)(1,2)(1.5,1)
    \region(3,1)(0.5,1)(0,2)(2,1)(2,1)
    \smallblack(3,2)
    \draw[ultra thick] \coord(0,2) -- \coord(1,1) -- \coord(2,1) --
    \coord(3,1) -- \coord(3,2);
    \path
    \coord(0.5,4) node {$(d)$}
    ;
  \end{scope}
  \begin{scope}[xshift=14.7cm]
    \rotation{-30}
    \grid{0}{6}
    \smallblack(0,2)
    \region(1,1)(0,1)(1,2)(1.5,1)(1.5,1)
    \region(2,1)(0.5,1)(0,2)(1,2)(1.5,1)
    \region(3,1)(0.5,1)(0,2)(1,2)(2,1)
    \draw[ultra thick] \coord(0,2) -- \coord(1,1) -- \coord(2,1) --
    \coord(3,1) -- \coord(3.75,0.25);
    \path
    \coord(0.5,4) node {$(e)$}
    ;
  \end{scope}
  \begin{scope}[xshift=1.4cm]
    \begin{scope}[xshift=0cm,yshift=-3cm]
      \rotation{-30}
      \grid{0}{6}
      \region(1,1)(0,1)(0,2)(1,2)(1.5,1)
      \region(2,1)(0.5,1)(0,2)(1,2)(1.5,1)
      \region(3,1)(0.5,1)(0,2)(1,2)(2,1)
      \draw[ultra thick] \coord(1,0.25) -- \coord(1,1) -- \coord(2,1) -- \coord(3,1) -- \coord(3.75,0.25);
      \path
      \coord(0.5,4) node {$(f)$}
      ;
    \end{scope}
    \begin{scope}[xshift=4.9cm,yshift=-3cm]
      \rotation{-30}
      \grid{0}{4}
      \region(1,1)(0,1)(0,2)(1,2)(1.5,1)
      \region(2,1)(0.5,1)(0,2)(1,2)(1.5,1)
      \region(3,1)(0.5,1)(0,2)(0,2)(1,1)
      \draw[ultra thick] \coord(1,0.25) -- \coord(1,1) -- \coord(2,1) -- \coord(3,1) -- \coord(3,1.75);
      \path
      \coord(-0.5,4) node {$(g)$}
      ;
    \end{scope}
    \begin{scope}[xshift=8.75cm,yshift=-3cm]
      \rotation{-30}
      \grid{0}{4}
      \region(1,1)(0,1)(1,2)(1.5,1)(1.5,1)
      \region(2,1)(0.5,1)(0,2)(1,2)(1.5,1)
      \region(3,1)(0.5,1)(0,2)(0,2)(1,1)
      \draw[ultra thick] \coord(0,2) -- \coord(1,1) -- \coord(2,1) -- \coord(3,1) -- \coord(3,1.75);
      \smallblack(0,2)
      \path
      \coord(-0.5,4) node {$(h)$}
      ;
    \end{scope}
    \begin{scope}[xshift=12.95cm,yshift=-3cm]
      \rotation{-30}
      \grid{0}{3}
      \region(-1,1)(1,1)(1,2)(1.5,1)(1.5,1)
      \region(0,1)(0.5,1)(0,2)(1,2)(1.5,1)
      \region(1,1)(0.5,1)(0,2)(1,2)(1.5,1)
      \region(2,1)(0.5,1)(0,2)(0,2)(1,1)
      \draw[ultra thick] \coord(-1.75,1.75) -- \coord(-1,1) -- \coord(0,1) -- \coord(1,1) -- \coord(2,1) -- \coord(2,1.75);
      \path
      \coord(-1,4) node {$(i)$}
      ;
    \end{scope}
  \end{scope}
\end{hexboard}
\]
\caption{Possible boundary components of winning paths within a
  triangular region of the unit grid.}
\label{fig:boundary}
\end{figure}
% ......................................................................

As can be seen by inspecting Figure~\ref{fig:boundary}, the union of
the domains of the stones in any one boundary component covers, by
area, one more downward-pointing than upward-pointing unit triangle.
From this we can get a relationship between several quantities.  For
the given triangular region of the unit grid, let $e$ denote the
excess number of downward-pointing over upward-pointing unit triangles
in the region. Let $b$ denote the number of boundary components of the
winning path in question.  Let $\tdown$ and $\tup$ denote the number
of downward- and upward-pointing wasted triangles in the region,
respectively. Since each boundary component reduces the excess by
exactly one, we have
\begin{equation}\label{eqn:ebt}
  e = b + \tdown - \tup.
\end{equation}

% ----------------------------------------------------------------------
\subsection{Improved bounds on the path length}

Using equations~\eqref{eqn:k} and {\eqref{eqn:ebt}} and
Lemma~\ref{lem:upward}, we can derive improved bounds on the length of
winning paths.

\begin{lemma}\label{lem:bound1}
  Assume $n\geq 6$, and let $k$ be the length of a winning
  path on a Hex board of size $n\times n$. Then $k\leq
  \frac{n^2}{2}-\frac{n}{4} + \frac{1}{4}$.
\end{lemma}

% ......................................................................
\begin{figure}
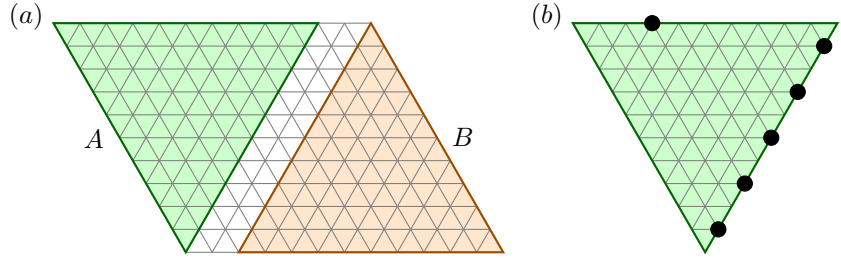

  \[
  (a)~
  \begin{hexboard}[scale=0.5,baseline={(0,0)}]
    \rotation{-30}
    \begin{pgfonlayer}{hexes}
      \fill[trianglelightgreen]
        \coord(0,1) -- \coord(0,11) -- \coord(10,1) -- cycle;
      \fill[trianglelightorange]
        \coord(12,11) -- \coord(12,1) -- \coord(2,11) -- cycle;
    \end{pgfonlayer}
    \unitgrid{11}{11}
    \draw[trianglegreen]
      \coord(0,1) -- \coord(0,11) -- \coord(10,1) -- cycle;
    \draw[triangleorange]
      \coord(12,11) -- \coord(12,1) -- \coord(2,11) -- cycle;
    \path \coord(-1,6.05) node {$A$};
    \path \coord(13,5.95) node {$B$};
  \end{hexboard}
  \quad
  (b)~
  \begin{hexboard}[scale=0.5,baseline={(0,0)}]
    \rotation{-30}
    \begin{pgfonlayer}{hexes}
      \fill[trianglelightgreen]
        \coord(0,1) -- \coord(0,11) -- \coord(10,1) -- cycle;
    \end{pgfonlayer}
    \unittriangle{11}
    \begin{pgfonlayer}{stones}
    \draw[trianglegreen]
      \coord(0,1) -- \coord(0,11) -- \coord(10,1) -- cycle;
    \end{pgfonlayer}
    \smallblack(1,10)
    \smallblack(3,8)  
    \smallblack(5,6)
    \smallblack(7,4)
    \smallblack(9,2)
    \smallblack(3,1)
  \end{hexboard}
  \]
  \caption{(a) The regions $A$ and $B$. (b) A winning path has at most
    $\frac{n+1}{2}$ boundary components on the boundary of region A.}
  \label{fig:ab}
\end{figure}
% ......................................................................

\begin{proof}
  Consider a winning path on a board of size $n\times n$. On the
  board's unit grid, let $A$ and $B$ be the maximal triangular regions
  containing the left and right acute corner, respectively, as shown
  in Figure~\ref{fig:ab}(a). We first consider region $A$. Within this
  region, the excess number of downward-pointing over upward-pointing
  unit triangles is $e = n-1$. As before, let $b$ be the number of
  boundary components of the winning path. There can be at most one
  boundary component along the region's top edge, and at most
  $\frac{n-1}{2}$ boundary components along its right edge, as shown
  in Figure~\ref{fig:ab}(b). Therefore, $b\leq 1 + \frac{n-1}{2} =
  \frac{n+1}{2}$.  As before, let $\tdown$ and $\tup$ be the number of
  downward- and upward-pointing wasted triangles in region $A$. Then
  using {\eqref{eqn:ebt}}, we have $\tdown - \tup = e - b \geq (n-1) -
  \frac{n+1}{2} = \frac{n}{2} - \frac{3}{2}$, or equivalently, $\tdown
  + \tup \geq \frac{n}{2} - \frac{3}{2} + 2\,\tup$.  Moreover, from
  Lemma~\ref{lem:upward} and our assumption that $n\geq 6$, we know that
  $\tup\geq 1$, which implies $\tdown+\tup \geq
  \frac{n}{2}+\frac{1}{2}$.  In other words, there are at least
  $\frac{n}{2}+\frac{1}{2}$ wasted triangles in region $A$. By the
  symmetric argument, there are also at least
  $\frac{n}{2}+\frac{1}{2}$ wasted triangles in region $B$, and
  therefore, the total number $t$ of wasted triangles in the entire
  unit grid satisfies $t \geq n+1$. Let $k$ be the length of the
  winning path; then by {\eqref{eqn:k}}, we have $k = \frac{n^2+1}{2}
  - \frac{t}{4} \leq \frac{n^2+1}{2} - \frac{n+1}{4} =
  \frac{n^2}{2}-\frac{n}{4}+\frac{1}{4}$, as claimed.
\end{proof}

The next lemma offers a small improvement over Lemma~\ref{lem:bound1}
in case $n\equiv 3\pmod{8}$.

\begin{lemma}\label{lem:bound2}
  Assume $n\geq 6$ and $n\equiv 3\pmod{8}$, and let $k$ be the length
  of a winning path on a Hex board of size $n\times n$. Then
  $k\leq \frac{n^2}{2}-\frac{n}{4} - \frac{3}{4}$.
\end{lemma}

\begin{proof}
  The proof is essentially the same as that of Lemma~\ref{lem:bound1},
  with one small twist. As before, we consider regions $A$ and $B$ as
  in Figure~\ref{fig:ab}(a). As before, we have $e=n-1$ and
  $b\leq\frac{n+1}{2}$. We claim that there are at least
  $\frac{n}{2}+\frac{3}{2}$ wasted triangles in region $A$. To prove
  this claim, we consider two cases.

  \begin{itemize}
    \item Case 1: $b < \frac{n+1}{2}$. Then $\tdown-\tup = e - b >
      \frac{n}{2} - \frac{3}{2}$, and with $\tup\geq 1$, this implies
      $\tdown+\tup > \frac{n}{2}+\frac{1}{2}$. Since
      $\frac{n}{2}+\frac{1}{2}$ is an integer, $\tdown+\tup$ must be
      at least the next largest integer, so $\tdown+\tup \geq
      \frac{n}{2}+\frac{3}{2}$, as claimed.
    \item Case 2: $b = \frac{n+1}{2}$. In this case, there must be
      exactly one boundary component along the top edge of region $A$,
      and exactly $\frac{n-1}{2}$ of them along the right edge, as
      shown in Figure~\ref{fig:ab}(b). Since there is no additional
      space, each boundary component must consist of a single stone,
      and therefore must be transversal.  The area of region
      $A$ is $(n-1)^2$ units, which is divisible by $4$. The area
      covered by the domains of stones is also divisible by $4$,
      because each interior stone has a domain of area $4$, 
      each boundary stone has a domain of area $2$, and there are
      an even number of boundary stones. Therefore, the number of
      wasted triangles within region $A$ must also be divisible by
      $4$. From the proof of Lemma~\ref{lem:bound1}, we know that
      there are at least $\frac{n}{2}+\frac{1}{2}$ wasted triangles in
      region $A$. Since this number is congruent to $2$ modulo $4$,
      there must be at least two more, so at least
      $\frac{n}{2}+\frac{5}{2}$ wasted triangles in region $A$. This
      implies the claim.
  \end{itemize}
  We have proved that the number of wasted triangles in region $A$ is
  at least $\frac{n}{2}+\frac{3}{2}$. Since by symmetry, the same
  applies to region $B$, the total number of wasted triangles on the
  whole grid is at least $n+3$, so by {\eqref{eqn:k}}, the winning
  path has length at most $\frac{n^2+1}{2} - \frac{n+3}{4} =
  \frac{n^2}{2} - \frac{n}{4} - \frac{1}{4}$. Moreover, since $n\equiv
  3\pmod{8}$, the latter quantity is not an integer; the largest
  integer below it is $\frac{n^2}{2} - \frac{n}{4} - \frac{3}{4}$,
  proving the lemma.
\end{proof}

The following theorem summarizes the findings of
Lemmas~\ref{lem:bound1} and {\ref{lem:bound2}}.

\begin{theorem}\label{thm:bounds}
  Assume $n\geq 6$, and let $k$ be the length of a winning path on a
  Hex board of size $n\times n$. Then we have:
  \begin{enumerate}\alphalabels
  \item If $n\equiv 0\pmod{4}$, then $k\leq\frac{n^2}{2}-\frac{n}{4}$.
  \item If $n\equiv 1\pmod{4}$, then $k\leq\frac{n^2}{2}-\frac{n}{4}-\frac{1}{4}$.
  \item If $n\equiv 2\pmod{4}$, then $k\leq\frac{n^2}{2}-\frac{n}{4}-\frac{1}{2}$.
  \item If $n\equiv 3\pmod{8}$, then $k\leq\frac{n^2}{2}-\frac{n}{4}-\frac{3}{4}$.
  \item If $n\equiv 7\pmod{8}$, then $k\leq\frac{n^2}{2}-\frac{n}{4}+\frac{1}{4}$.
  \end{enumerate}
\end{theorem}

\begin{proof}
  Properties (a), (b), (c), and (e) are direct consequences of
  Lemma~\ref{lem:bound1}, because in each case, the claimed bound is
  the largest integer below $\frac{n^2}{2}-\frac{n}{4} +
  \frac{1}{4}$. Property (d) is Lemma~\ref{lem:bound2}.
\end{proof}

% ----------------------------------------------------------------------
\section{Sharpness of the bounds}

\begin{theorem}
  The bounds in Theorem~\ref{thm:bounds} are sharp for all $n\geq 6$
  except for $n=9$.
\end{theorem}

\begin{proof}
  We first consider the case $n\leq 20$. Figure~\ref{fig:witnesses}
  shows a winning path for each $n=1,\ldots,20$.  The length of each
  path is listed in Figure~\ref{fig:path-lengths}, along with each
  applicable bound from Theorem~\ref{thm:bounds}. These paths witness
  the fact that the bounds are sharp except for $n=9$. Conversely, the
  bounds of Theorem~\ref{thm:bounds} imply that the paths in
  Figure~\ref{fig:witnesses} are optimal, except perhaps for
  $n=1,2,3,4,5$ and $n=9$. The optimality of the paths for these
  remaining cases can be shown by exhaustive search.

  Next, we consider the case $n>20$. Figure~\ref{fig:witnesses-extend}
  shows three constructions for turning a winning path for board size
  $n$ into a winning path for board size $n+8$. By applying these
  constructions recursively, starting with the winning paths for
  $n=13,\ldots,20$ from Figure~\ref{fig:witnesses}, we obtain winning
  paths for all $n>20$. As a matter of fact, the paths for
  $n=13,14,15,16,18,19,20$ in Figure~\ref{fig:witnesses} were already
  obtained by this method.

  It remains to show that the resulting paths match the bounds of
  Theorem~\ref{thm:bounds}. It is easy to check that given a winning
  path of length $k$ for board size $n$, the constructions of
  Figure~\ref{fig:witnesses-extend} yield a winning path for board
  size $n+8$ of length $k'=k+8n+30$. (For $n=6$, this can be checked
  by counting the stones in Figure~\ref{fig:witnesses-extend}. Then
  observe that each pattern on the right-hand side of
  Figure~\ref{fig:witnesses-extend} has 2 parallel bands of stones
  running along each of the four board edges. Therefore, each time we
  increase the size of the pattern by one, its number of stones
  increases by 8.)

  Now suppose that $k=\frac{n^2}{2}-\frac{n}{4}+c$. Then it follows
  that $k'=k+8n+30=\frac{(n+8)^2}{2}-\frac{n+8}{4}+c$. Therefore, if a
  path for board size $n$ matches its bound of
  Theorem~\ref{thm:bounds}, then so does the constructed path for
  board size $n+8$. The result follows by induction.
\end{proof}

\begin{remark}
  One can also ask \emph{how many} optimal winning paths there are for
  each board size. The number of optimal winning paths for
  $n=1,\ldots,20$ can be calculated by enumeration, and is shown in
  Figure~\ref{fig:path-lengths}. No formula is currently
  known for these numbers.

  One may note that, compared to other nearby $n$, the number of
  solutions is unusually large for $n=9$, $n=11$, and $n=19$. This may
  be related to the fact that the quantity
  $\frac{n^2}{2}-\frac{n}{4}-k$ is slightly larger than usual in these
  cases, where $k$ is the path length. This yields an above-average
  number of wasted triangles and therefore more potential freedom in
  choosing winning paths. On the other hand, the number of solutions
  is unusually small for $n=10$ and $n=12$, perhaps reflecting the
  fact that these solutions are qualitatively different from those of
  smaller board sizes.

  A possible intuitive explanation for the discontinuities in the path
  counts of Figure~\ref{fig:path-lengths} is that we are only
  reporting the number of longest paths, rather than all paths. So if
  there were, say, 1000 paths of length $k$ and $1$ path of length
  $k+1$, we would report the number ``1'', but if instead, there were
  $0$ paths of length $k+1$, we would report ``1000''. Although the
  difference between 1 solution and 0 solutions is very small, it
  would cause a large change in what is reported in
  Figure~\ref{fig:path-lengths}.
\end{remark}

\begin{remark}
  All of the winning paths shown in Figures~\ref{fig:witnesses} and
  {\ref{fig:witnesses-extend}} start and end in a corner of the
  board. However, not all optimal winning paths have this
  property. For example, only 2 of the 12 optimal winning paths for
  $10\times 10$ start and end in a corner. Here are two optimal paths
  that do not go corner-to-corner:
  
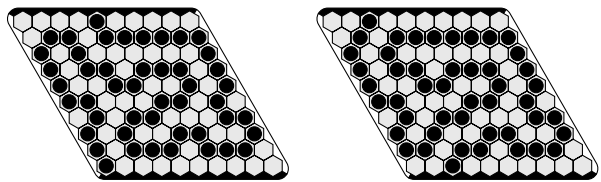
\[
  \begin{hexboard}[scale=0.35] % 10x10
    \rotation{-30}
    \board(10,10)
    \black(5,1)\black(1,3)\black(2,2)\black(3,2)\black(1,4)
    \black(5,2)\black(6,2)\black(7,2)\black(8,2)\black(9,2)
    \black(10,2)\black(10,3)\black(3,3)\black(3,4)\black(4,4)
    \black(6,4)\black(7,4)\black(8,4)\black(10,4)\black(1,5)
    \black(4,5)\black(5,5)\black(8,5)\black(10,5)\black(1,6)
    \black(2,6)\black(6,6)\black(7,6)\black(9,6)\black(2,7)
    \black(4,7)\black(5,7)\black(9,7)\black(10,7)\black(1,8)
    \black(3,8)\black(6,8)\black(7,8)\black(10,8)\black(1,9)
    \black(3,9)\black(4,9)\black(5,9)\black(7,9)\black(8,9)
    \black(9,9)\black(1,10)
  \end{hexboard}
  \quad
  \begin{hexboard}[scale=0.35] % 10x10
    \rotation{-30}
    \board(10,10)
    \black(3,1)\black(1,3)\black(2,2)\black(4,2)\black(1,4)
    \black(5,2)\black(6,2)\black(7,2)\black(8,2)\black(9,2)
    \black(10,2)\black(10,3)\black(3,3)\black(3,4)\black(4,4)
    \black(6,4)\black(7,4)\black(8,4)\black(10,4)\black(1,5)
    \black(4,5)\black(5,5)\black(8,5)\black(10,5)\black(1,6)
    \black(2,6)\black(6,6)\black(7,6)\black(9,6)\black(2,7)
    \black(4,7)\black(5,7)\black(9,7)\black(10,7)\black(1,8)
    \black(3,8)\black(6,8)\black(7,8)\black(10,8)\black(1,9)
    \black(3,10)\black(4,9)\black(5,9)\black(7,9)\black(8,9)
    \black(9,9)\black(2,9)
  \end{hexboard}
  \]
\end{remark}

\begin{remark}
  Manually finding an optimal winning path can be a fun and difficult
  puzzle. Probably the most difficult cases are $n=10$, $n=12$, and
  $n=17$, because the solutions do not seem to follow any pattern that
  can be guessed from smaller board sizes.
\end{remark}

% ......................................................................
\begin{figure}[p]
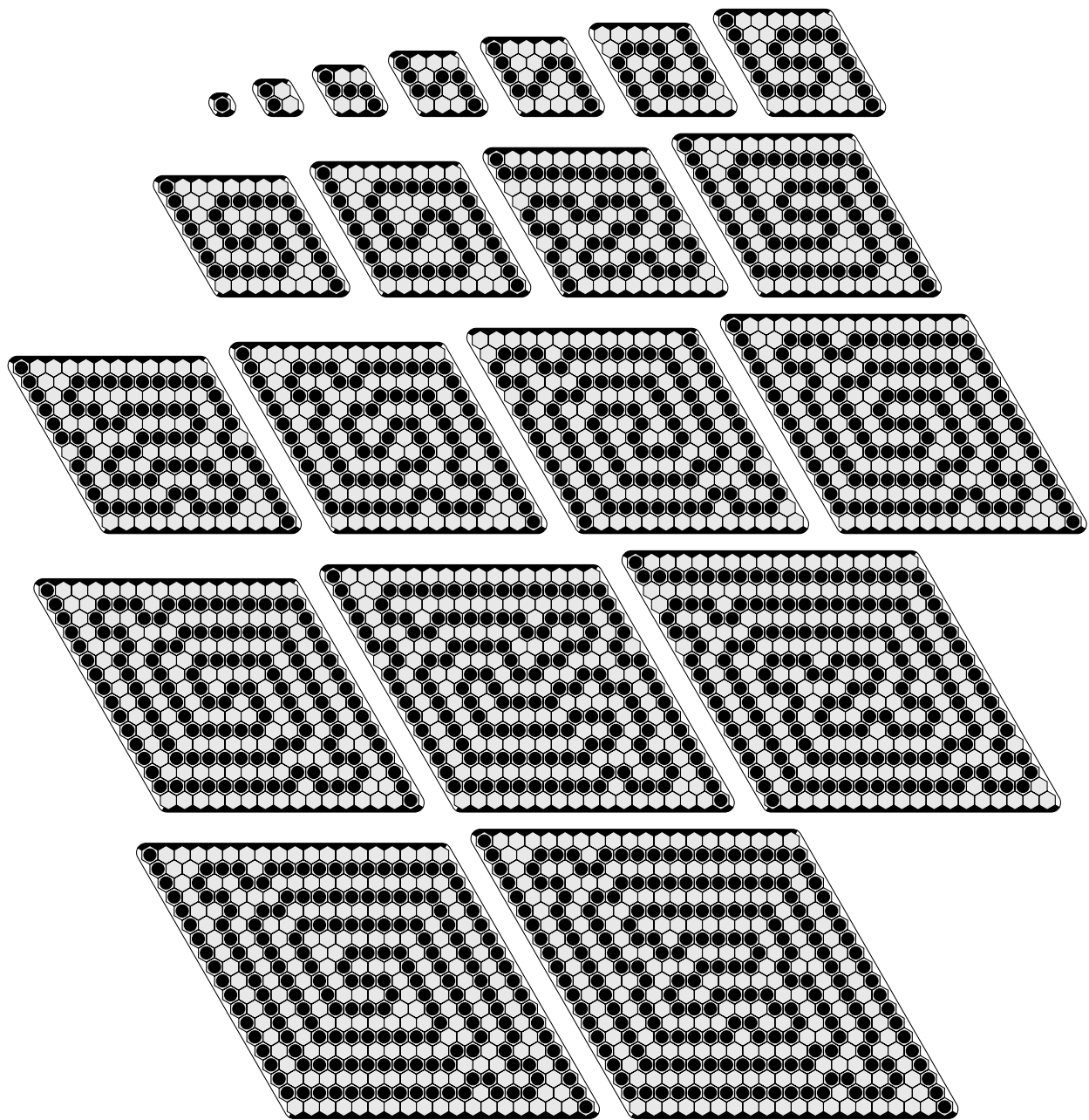

  \newlength{\increment}
  \setlength{\increment}{0.1225cm}
  \newcommand{\offset}{\hspace{0.65em}}
  \[
\begin{hexboard}[scale=0.35] % 1x1
  \rotation{-30}
  \board(1,1)
  \black(1,1)
  %\path \coord(1,2) node {$n=1$, length 1};
\end{hexboard}
\offset\hspace{0\increment}
\begin{hexboard}[scale=0.35] % 2x2
  \rotation{-30}
  \board(2,2)
  \black(1,1)\black(1,2)
  %\path \coord(1,3) node {$n=2$, length 2};
\end{hexboard}
\offset\hspace{-1\increment}
\begin{hexboard}[scale=0.35] % 3x3
  \rotation{-30}
  \board(3,3)
  \black(1,1)\black(1,2)\black(2,2)\black(3,2)\black(3,3)
  %\path \coord(1,4) node {$n=3$, length 5};
\end{hexboard}
\offset\hspace{-2\increment}
\begin{hexboard}[scale=0.35] % 4x4
  \rotation{-30}
  \board(4,4)
  \black(1,1)\black(1,2)\black(3,2)\black(4,2)\black(1,3)
  \black(2,3)\black(4,3)\black(4,4)
\end{hexboard}
\offset\hspace{-3\increment}
\begin{hexboard}[scale=0.35] % 5x5
  \rotation{-30}
  \board(5,5)
  \black(1,1)\black(1,2)\black(4,2)\black(5,2)\black(1,3)
  \black(3,3)\black(5,3)\black(1,4)\black(2,4)\black(5,4)
  \black(5,5)
\end{hexboard}
\offset\hspace{-4\increment}
\begin{hexboard}[scale=0.35] % 6x6
  \rotation{-30}
  \board(6,6)
  \black(6,1)\black(2,2)\black(3,2)\black(4,2)\black(6,2)
  \black(1,3)\black(4,3)\black(6,3)\black(1,4)\black(3,4)
  \black(6,4)\black(1,5)\black(3,5)\black(4,5)\black(5,5)
  \black(1,6)
\end{hexboard}
\offset\hspace{-5\increment}
\begin{hexboard}[scale=0.35] % 7x7
  \rotation{-30}
  \board(7,7)
  \black(1,1)\black(1,2)\black(4,2)\black(5,2)\black(6,2)
  \black(7,2)\black(1,3)\black(3,3)\black(7,3)\black(1,4)
  \black(3,4)\black(4,4)\black(5,4)\black(7,4)\black(1,5)
  \black(5,5)\black(7,5)\black(1,6)\black(2,6)\black(3,6)
  \black(4,6)\black(7,6)\black(7,7)
\end{hexboard}
\]
\[
\begin{hexboard}[scale=0.35] % 8x8
  \rotation{-30}
  \board(8,8)
  \black(1,1)\black(1,2)\black(4,2)\black(5,2)\black(6,2)
  \black(7,2)\black(8,2)\black(1,3)\black(3,3)\black(8,3)
  \black(1,4)\black(3,4)\black(5,4)\black(6,4)\black(8,4)
  \black(1,5)\black(3,5)\black(4,5)\black(6,5)\black(8,5)
  \black(1,6)\black(6,6)\black(8,6)\black(1,7)\black(2,7)
  \black(3,7)\black(4,7)\black(5,7)\black(8,7)\black(8,8)
\end{hexboard}
\offset\hspace{-7\increment}
\begin{hexboard}[scale=0.35] % 9x9
  \rotation{-30}
  \board(9,9)
  \black(1,1)\black(1,2)\black(4,2)\black(5,2)\black(6,2)
  \black(7,2)\black(8,2)\black(9,2)\black(1,3)\black(3,3)
  \black(9,3)\black(1,4)\black(3,4)\black(6,4)\black(7,4)
  \black(9,4)\black(1,5)\black(3,5)\black(5,5)\black(7,5)
  \black(9,5)\black(1,6)\black(3,6)\black(4,6)\black(7,6)
  \black(9,6)\black(1,7)\black(7,7)\black(9,7)\black(1,8)
  \black(2,8)\black(3,8)\black(4,8)\black(5,8)\black(6,8)
  \black(9,8)\black(9,9)
\end{hexboard}
\offset\hspace{-8\increment}
\begin{hexboard}[scale=0.35] % 10x10
  \rotation{-30}
  \board(10,10)
  \black(1,1)\black(1,2)\black(2,2)\black(3,2)\black(4,2)
  \black(5,2)\black(6,2)\black(7,2)\black(8,2)\black(9,2)
  \black(10,2)\black(10,3)\black(2,4)\black(3,4)\black(4,4)
  \black(6,4)\black(7,4)\black(8,4)\black(10,4)\black(1,5)
  \black(4,5)\black(5,5)\black(8,5)\black(10,5)\black(1,6)
  \black(2,6)\black(6,6)\black(7,6)\black(9,6)\black(2,7)
  \black(4,7)\black(5,7)\black(9,7)\black(10,7)\black(1,8)
  \black(3,8)\black(6,8)\black(7,8)\black(10,8)\black(1,9)
  \black(3,9)\black(4,9)\black(5,9)\black(7,9)\black(8,9)
  \black(9,9)\black(1,10)
\end{hexboard}
\offset\hspace{-9\increment}
\begin{hexboard}[scale=0.35] % 11x11
  \rotation{-30}
  \board(11,11)
  \black(1,1)\black(1,2)\black(4,2)\black(5,2)\black(6,2)
  \black(7,2)\black(8,2)\black(9,2)\black(10,2)\black(11,2)
  \black(1,3)\black(3,3)\black(11,3)\black(1,4)\black(3,4)
  \black(6,4)\black(7,4)\black(8,4)\black(9,4)\black(11,4)
  \black(1,5)\black(3,5)\black(5,5)\black(9,5)\black(11,5)
  \black(1,6)\black(3,6)\black(5,6)\black(6,6)\black(7,6)
  \black(9,6)\black(11,6)\black(1,7)\black(3,7)\black(7,7)
  \black(9,7)\black(11,7)\black(1,8)\black(3,8)\black(4,8)
  \black(5,8)\black(6,8)\black(9,8)\black(11,8)\black(1,9)
  \black(9,9)\black(11,9)\black(1,10)\black(2,10)\black(3,10)
  \black(4,10)\black(5,10)\black(6,10)\black(7,10)\black(8,10)
  \black(11,10)\black(11,11)
\end{hexboard}
\]
\[
\begin{hexboard}[scale=0.35] % 12x12
  \rotation{-30}
  \board(12,12)
  \black(1,1)\black(1,2)\black(4,2)\black(5,2)\black(6,2)
  \black(7,2)\black(8,2)\black(9,2)\black(10,2)\black(11,2)
  \black(12,2)\black(1,3)\black(3,3)\black(12,3)\black(1,4)
  \black(3,4)\black(4,4)\black(6,4)\black(7,4)\black(8,4)
  \black(9,4)\black(10,4)\black(12,4)\black(1,5)\black(4,5)
  \black(5,5)\black(10,5)\black(12,5)\black(1,6)\black(2,6)
  \black(6,6)\black(7,6)\black(8,6)\black(9,6)\black(11,6)
  \black(2,7)\black(4,7)\black(5,7)\black(11,7)\black(12,7)
  \black(1,8)\black(3,8)\black(6,8)\black(7,8)\black(8,8)
  \black(9,8)\black(12,8)\black(1,9)\black(3,9)\black(4,9)
  \black(5,9)\black(9,9)\black(10,9)\black(12,9)\black(1,10)
  \black(6,10)\black(7,10)\black(10,10)\black(12,10)\black(1,11)
  \black(2,11)\black(3,11)\black(4,11)\black(5,11)\black(7,11)
  \black(8,11)\black(9,11)\black(12,11)\black(12,12)
\end{hexboard}
\offset\hspace{-11\increment}
\begin{hexboard}[scale=0.35] % 13x13
  \rotation{-30}
  \board(13,13)
  \black(1,1)\black(1,2)\black(4,2)\black(5,2)\black(6,2)
  \black(8,2)\black(9,2)\black(10,2)\black(11,2)\black(12,2)
  \black(13,2)\black(1,3)\black(3,3)\black(6,3)\black(7,3)
  \black(13,3)\black(1,4)\black(3,4)\black(4,4)\black(8,4)
  \black(9,4)\black(10,4)\black(11,4)\black(13,4)\black(1,5)
  \black(4,5)\black(6,5)\black(7,5)\black(11,5)\black(13,5)
  \black(1,6)\black(3,6)\black(5,6)\black(8,6)\black(9,6)
  \black(11,6)\black(13,6)\black(1,7)\black(3,7)\black(5,7)
  \black(7,7)\black(9,7)\black(11,7)\black(13,7)\black(1,8)
  \black(3,8)\black(5,8)\black(6,8)\black(9,8)\black(11,8)
  \black(13,8)\black(1,9)\black(3,9)\black(7,9)\black(8,9)
  \black(10,9)\black(13,9)\black(1,10)\black(3,10)\black(4,10)
  \black(5,10)\black(6,10)\black(10,10)\black(11,10)\black(13,10)
  \black(1,11)\black(7,11)\black(8,11)\black(11,11)\black(13,11)
  \black(1,12)\black(2,12)\black(3,12)\black(4,12)\black(5,12)
  \black(6,12)\black(8,12)\black(9,12)\black(10,12)\black(13,12)
  \black(13,13)
\end{hexboard}
\offset\hspace{-12\increment}
\begin{hexboard}[scale=0.35] % 14x14
  \rotation{-30}
  \board(14,14)
  \black(14,1)\black(2,2)\black(3,2)\black(4,2)\black(6,2)
  \black(7,2)\black(8,2)\black(9,2)\black(10,2)\black(11,2)
  \black(12,2)\black(14,2)\black(1,3)\black(4,3)\black(5,3)
  \black(12,3)\black(14,3)\black(1,4)\black(2,4)\black(6,4)
  \black(7,4)\black(8,4)\black(9,4)\black(10,4)\black(12,4)
  \black(14,4)\black(2,5)\black(4,5)\black(5,5)\black(10,5)
  \black(12,5)\black(14,5)\black(1,6)\black(3,6)\black(6,6)
  \black(7,6)\black(8,6)\black(10,6)\black(12,6)\black(14,6)
  \black(1,7)\black(3,7)\black(5,7)\black(8,7)\black(10,7)
  \black(12,7)\black(14,7)\black(1,8)\black(3,8)\black(5,8)
  \black(7,8)\black(10,8)\black(12,8)\black(14,8)\black(1,9)
  \black(3,9)\black(5,9)\black(7,9)\black(8,9)\black(9,9)
  \black(12,9)\black(14,9)\black(1,10)\black(3,10)\black(5,10)
  \black(10,10)\black(11,10)\black(13,10)\black(1,11)\black(3,11)
  \black(5,11)\black(6,11)\black(7,11)\black(8,11)\black(9,11)
  \black(13,11)\black(14,11)\black(1,12)\black(3,12)\black(10,12)
  \black(11,12)\black(14,12)\black(1,13)\black(3,13)\black(4,13)
  \black(5,13)\black(6,13)\black(7,13)\black(8,13)\black(9,13)
  \black(11,13)\black(12,13)\black(13,13)\black(1,14)
\end{hexboard}
\offset\hspace{-13\increment}
\begin{hexboard}[scale=0.35] % 15x15
  \rotation{-30}
  \board(15,15)
  \black(1,1)\black(1,2)\black(4,2)\black(5,2)\black(6,2)
  \black(8,2)\black(9,2)\black(10,2)\black(11,2)\black(12,2)
  \black(13,2)\black(14,2)\black(15,2)\black(1,3)\black(3,3)
  \black(6,3)\black(7,3)\black(15,3)\black(1,4)\black(3,4)
  \black(4,4)\black(8,4)\black(9,4)\black(10,4)\black(11,4)
  \black(12,4)\black(13,4)\black(15,4)\black(1,5)\black(4,5)
  \black(6,5)\black(7,5)\black(13,5)\black(15,5)\black(1,6)
  \black(3,6)\black(5,6)\black(8,6)\black(9,6)\black(10,6)
  \black(11,6)\black(13,6)\black(15,6)\black(1,7)\black(3,7)
  \black(5,7)\black(7,7)\black(11,7)\black(13,7)\black(15,7)
  \black(1,8)\black(3,8)\black(5,8)\black(7,8)\black(8,8)
  \black(9,8)\black(11,8)\black(13,8)\black(15,8)\black(1,9)
  \black(3,9)\black(5,9)\black(9,9)\black(11,9)\black(13,9)
  \black(15,9)\black(1,10)\black(3,10)\black(5,10)\black(6,10)
  \black(7,10)\black(8,10)\black(11,10)\black(13,10)\black(15,10)
  \black(1,11)\black(3,11)\black(9,11)\black(10,11)\black(12,11)
  \black(15,11)\black(1,12)\black(3,12)\black(4,12)\black(5,12)
  \black(6,12)\black(7,12)\black(8,12)\black(12,12)\black(13,12)
  \black(15,12)\black(1,13)\black(9,13)\black(10,13)\black(13,13)
  \black(15,13)\black(1,14)\black(2,14)\black(3,14)\black(4,14)
  \black(5,14)\black(6,14)\black(7,14)\black(8,14)\black(10,14)
  \black(11,14)\black(12,14)\black(15,14)\black(15,15)
\end{hexboard}
\]
\[
\begin{hexboard}[scale=0.35] % 16x16
  \rotation{-30}
  \board(16,16)
  \black(1,1)\black(1,2)\black(4,2)\black(5,2)\black(6,2)
  \black(8,2)\black(9,2)\black(10,2)\black(11,2)\black(12,2)
  \black(13,2)\black(14,2)\black(15,2)\black(16,2)\black(1,3)
  \black(3,3)\black(6,3)\black(7,3)\black(16,3)\black(1,4)
  \black(3,4)\black(4,4)\black(8,4)\black(9,4)\black(10,4)
  \black(11,4)\black(12,4)\black(13,4)\black(14,4)\black(16,4)
  \black(1,5)\black(4,5)\black(6,5)\black(7,5)\black(14,5)
  \black(16,5)\black(1,6)\black(3,6)\black(5,6)\black(8,6)
  \black(9,6)\black(10,6)\black(11,6)\black(12,6)\black(14,6)
  \black(16,6)\black(1,7)\black(3,7)\black(5,7)\black(7,7)
  \black(12,7)\black(14,7)\black(16,7)\black(1,8)\black(3,8)
  \black(5,8)\black(7,8)\black(9,8)\black(10,8)\black(12,8)
  \black(14,8)\black(16,8)\black(1,9)\black(3,9)\black(5,9)
  \black(7,9)\black(8,9)\black(10,9)\black(12,9)\black(14,9)
  \black(16,9)\black(1,10)\black(3,10)\black(5,10)\black(10,10)
  \black(12,10)\black(14,10)\black(16,10)\black(1,11)\black(3,11)
  \black(5,11)\black(6,11)\black(7,11)\black(8,11)\black(9,11)
  \black(12,11)\black(14,11)\black(16,11)\black(1,12)\black(3,12)
  \black(10,12)\black(11,12)\black(13,12)\black(16,12)\black(1,13)
  \black(3,13)\black(4,13)\black(5,13)\black(6,13)\black(7,13)
  \black(8,13)\black(9,13)\black(13,13)\black(14,13)\black(16,13)
  \black(1,14)\black(10,14)\black(11,14)\black(14,14)\black(16,14)
  \black(1,15)\black(2,15)\black(3,15)\black(4,15)\black(5,15)
  \black(6,15)\black(7,15)\black(8,15)\black(9,15)\black(11,15)
  \black(12,15)\black(13,15)\black(16,15)\black(16,16)
\end{hexboard}
\offset\hspace{-15\increment}
\begin{hexboard}[scale=0.35] % 17x17
  \rotation{-30}
  \board(17,17)
  \black(1,1)\black(1,2)\black(4,2)\black(5,2)\black(6,2)
  \black(7,2)\black(8,2)\black(9,2)\black(10,2)\black(11,2)
  \black(12,2)\black(13,2)\black(14,2)\black(15,2)\black(16,2)
  \black(17,2)\black(1,3)\black(3,3)\black(17,3)\black(1,4)
  \black(3,4)\black(4,4)\black(6,4)\black(7,4)\black(8,4)
  \black(9,4)\black(10,4)\black(11,4)\black(13,4)\black(14,4)
  \black(15,4)\black(17,4)\black(1,5)\black(4,5)\black(5,5)
  \black(11,5)\black(12,5)\black(15,5)\black(17,5)\black(1,6)
  \black(2,6)\black(6,6)\black(7,6)\black(8,6)\black(9,6)
  \black(13,6)\black(14,6)\black(16,6)\black(2,7)\black(4,7)
  \black(5,7)\black(9,7)\black(11,7)\black(12,7)\black(16,7)
  \black(17,7)\black(1,8)\black(3,8)\black(6,8)\black(7,8)
  \black(9,8)\black(10,8)\black(13,8)\black(14,8)\black(17,8)
  \black(1,9)\black(3,9)\black(5,9)\black(7,9)\black(11,9)
  \black(12,9)\black(14,9)\black(15,9)\black(17,9)\black(1,10)
  \black(3,10)\black(5,10)\black(7,10)\black(8,10)\black(9,10)
  \black(10,10)\black(15,10)\black(17,10)\black(1,11)\black(3,11)
  \black(5,11)\black(11,11)\black(12,11)\black(13,11)\black(15,11)
  \black(17,11)\black(1,12)\black(3,12)\black(5,12)\black(6,12)
  \black(7,12)\black(8,12)\black(9,12)\black(10,12)\black(13,12)
  \black(15,12)\black(17,12)\black(1,13)\black(3,13)\black(11,13)
  \black(12,13)\black(14,13)\black(17,13)\black(1,14)\black(3,14)
  \black(4,14)\black(5,14)\black(6,14)\black(7,14)\black(8,14)
  \black(9,14)\black(10,14)\black(14,14)\black(15,14)\black(17,14)
  \black(1,15)\black(11,15)\black(12,15)\black(15,15)\black(17,15)
  \black(1,16)\black(2,16)\black(3,16)\black(4,16)\black(5,16)
  \black(6,16)\black(7,16)\black(8,16)\black(9,16)\black(10,16)
  \black(12,16)\black(13,16)\black(14,16)\black(17,16)
  \black(17,17)
\end{hexboard}
\offset\hspace{-16\increment}
\begin{hexboard}[scale=0.35] % 18x18
  \rotation{-30}
  \board(18,18)
  \black(1,1)\black(1,2)\black(2,2)\black(3,2)\black(4,2)
  \black(5,2)\black(6,2)\black(7,2)\black(8,2)\black(9,2)
  \black(10,2)\black(11,2)\black(12,2)\black(13,2)\black(14,2)
  \black(15,2)\black(16,2)\black(17,2)\black(18,2)\black(18,3)
  \black(2,4)\black(3,4)\black(4,4)\black(6,4)\black(7,4)
  \black(8,4)\black(9,4)\black(10,4)\black(11,4)\black(12,4)
  \black(13,4)\black(14,4)\black(15,4)\black(16,4)\black(18,4)
  \black(1,5)\black(4,5)\black(5,5)\black(16,5)\black(18,5)
  \black(1,6)\black(2,6)\black(6,6)\black(7,6)\black(8,6)
  \black(9,6)\black(10,6)\black(11,6)\black(12,6)\black(13,6)
  \black(14,6)\black(16,6)\black(18,6)\black(2,7)\black(4,7)
  \black(5,7)\black(14,7)\black(16,7)\black(18,7)\black(1,8)
  \black(3,8)\black(6,8)\black(7,8)\black(8,8)\black(10,8)
  \black(11,8)\black(12,8)\black(14,8)\black(16,8)\black(18,8)
  \black(1,9)\black(3,9)\black(5,9)\black(8,9)\black(9,9)
  \black(12,9)\black(14,9)\black(16,9)\black(18,9)\black(1,10)
  \black(3,10)\black(5,10)\black(6,10)\black(10,10)\black(11,10)
  \black(13,10)\black(16,10)\black(18,10)\black(1,11)\black(3,11)
  \black(6,11)\black(8,11)\black(9,11)\black(13,11)\black(14,11)
  \black(16,11)\black(18,11)\black(1,12)\black(3,12)\black(5,12)
  \black(7,12)\black(10,12)\black(11,12)\black(14,12)\black(16,12)
  \black(18,12)\black(1,13)\black(3,13)\black(5,13)\black(7,13)
  \black(8,13)\black(9,13)\black(11,13)\black(12,13)\black(13,13)
  \black(16,13)\black(18,13)\black(1,14)\black(3,14)\black(5,14)
  \black(14,14)\black(15,14)\black(17,14)\black(1,15)\black(3,15)
  \black(5,15)\black(6,15)\black(7,15)\black(8,15)\black(9,15)
  \black(10,15)\black(11,15)\black(12,15)\black(13,15)
  \black(17,15)\black(18,15)\black(1,16)\black(3,16)\black(14,16)
  \black(15,16)\black(18,16)\black(1,17)\black(3,17)\black(4,17)
  \black(5,17)\black(6,17)\black(7,17)\black(8,17)\black(9,17)
  \black(10,17)\black(11,17)\black(12,17)\black(13,17)
  \black(15,17)\black(16,17)\black(17,17)\black(1,18)
\end{hexboard}
\]
\[
\begin{hexboard}[scale=0.35] % 19x19
  \rotation{-30}
  \board(19,19)
  \black(1,1)\black(1,2)\black(4,2)\black(5,2)\black(6,2)
  \black(8,2)\black(9,2)\black(10,2)\black(11,2)\black(12,2)
  \black(13,2)\black(14,2)\black(15,2)\black(16,2)\black(17,2)
  \black(18,2)\black(19,2)\black(1,3)\black(3,3)\black(6,3)
  \black(7,3)\black(19,3)\black(1,4)\black(3,4)\black(4,4)
  \black(8,4)\black(9,4)\black(10,4)\black(11,4)\black(12,4)
  \black(13,4)\black(14,4)\black(15,4)\black(16,4)\black(17,4)
  \black(19,4)\black(1,5)\black(4,5)\black(6,5)\black(7,5)
  \black(17,5)\black(19,5)\black(1,6)\black(3,6)\black(5,6)
  \black(8,6)\black(9,6)\black(10,6)\black(11,6)\black(12,6)
  \black(13,6)\black(14,6)\black(15,6)\black(17,6)\black(19,6)
  \black(1,7)\black(3,7)\black(5,7)\black(7,7)\black(15,7)
  \black(17,7)\black(19,7)\black(1,8)\black(3,8)\black(5,8)
  \black(7,8)\black(10,8)\black(11,8)\black(12,8)\black(13,8)
  \black(15,8)\black(17,8)\black(19,8)\black(1,9)\black(3,9)
  \black(5,9)\black(7,9)\black(9,9)\black(13,9)\black(15,9)
  \black(17,9)\black(19,9)\black(1,10)\black(3,10)\black(5,10)
  \black(7,10)\black(9,10)\black(10,10)\black(11,10)\black(13,10)
  \black(15,10)\black(17,10)\black(19,10)\black(1,11)\black(3,11)
  \black(5,11)\black(7,11)\black(11,11)\black(13,11)\black(15,11)
  \black(17,11)\black(19,11)\black(1,12)\black(3,12)\black(5,12)
  \black(7,12)\black(8,12)\black(9,12)\black(10,12)\black(13,12)
  \black(15,12)\black(17,12)\black(19,12)\black(1,13)\black(3,13)
  \black(5,13)\black(13,13)\black(15,13)\black(17,13)\black(19,13)
  \black(1,14)\black(3,14)\black(5,14)\black(6,14)\black(7,14)
  \black(8,14)\black(9,14)\black(10,14)\black(11,14)\black(12,14)
  \black(15,14)\black(17,14)\black(19,14)\black(1,15)\black(3,15)
  \black(13,15)\black(14,15)\black(16,15)\black(19,15)\black(1,16)
  \black(3,16)\black(4,16)\black(5,16)\black(6,16)\black(7,16)
  \black(8,16)\black(9,16)\black(10,16)\black(11,16)\black(12,16)
  \black(16,16)\black(17,16)\black(19,16)\black(1,17)\black(13,17)
  \black(14,17)\black(17,17)\black(19,17)\black(1,18)\black(2,18)
  \black(3,18)\black(4,18)\black(5,18)\black(6,18)\black(7,18)
  \black(8,18)\black(9,18)\black(10,18)\black(11,18)\black(12,18)
  \black(14,18)\black(15,18)\black(16,18)\black(19,18)
  \black(19,19)
\end{hexboard}
\offset\hspace{-18\increment}
\begin{hexboard}[scale=0.35] % 20x20
  \rotation{-30}
  \board(20,20)
  \black(1,1)\black(1,2)\black(4,2)\black(5,2)\black(6,2)
  \black(8,2)\black(9,2)\black(10,2)\black(11,2)\black(12,2)
  \black(13,2)\black(14,2)\black(15,2)\black(16,2)\black(17,2)
  \black(18,2)\black(19,2)\black(20,2)\black(1,3)\black(3,3)
  \black(6,3)\black(7,3)\black(20,3)\black(1,4)\black(3,4)
  \black(4,4)\black(8,4)\black(9,4)\black(10,4)\black(11,4)
  \black(12,4)\black(13,4)\black(14,4)\black(15,4)\black(16,4)
  \black(17,4)\black(18,4)\black(20,4)\black(1,5)\black(4,5)
  \black(6,5)\black(7,5)\black(18,5)\black(20,5)\black(1,6)
  \black(3,6)\black(5,6)\black(8,6)\black(9,6)\black(10,6)
  \black(11,6)\black(12,6)\black(13,6)\black(14,6)\black(15,6)
  \black(16,6)\black(18,6)\black(20,6)\black(1,7)\black(3,7)
  \black(5,7)\black(7,7)\black(16,7)\black(18,7)\black(20,7)
  \black(1,8)\black(3,8)\black(5,8)\black(7,8)\black(8,8)
  \black(10,8)\black(11,8)\black(12,8)\black(13,8)\black(14,8)
  \black(16,8)\black(18,8)\black(20,8)\black(1,9)\black(3,9)
  \black(5,9)\black(8,9)\black(9,9)\black(14,9)\black(16,9)
  \black(18,9)\black(20,9)\black(1,10)\black(3,10)\black(5,10)
  \black(6,10)\black(10,10)\black(11,10)\black(12,10)\black(13,10)
  \black(15,10)\black(18,10)\black(20,10)\black(1,11)\black(3,11)
  \black(6,11)\black(8,11)\black(9,11)\black(15,11)\black(16,11)
  \black(18,11)\black(20,11)\black(1,12)\black(3,12)\black(5,12)
  \black(7,12)\black(10,12)\black(11,12)\black(12,12)\black(13,12)
  \black(16,12)\black(18,12)\black(20,12)\black(1,13)\black(3,13)
  \black(5,13)\black(7,13)\black(8,13)\black(9,13)\black(13,13)
  \black(14,13)\black(16,13)\black(18,13)\black(20,13)\black(1,14)
  \black(3,14)\black(5,14)\black(10,14)\black(11,14)\black(14,14)
  \black(16,14)\black(18,14)\black(20,14)\black(1,15)\black(3,15)
  \black(5,15)\black(6,15)\black(7,15)\black(8,15)\black(9,15)
  \black(11,15)\black(12,15)\black(13,15)\black(16,15)
  \black(18,15)\black(20,15)\black(1,16)\black(3,16)\black(14,16)
  \black(15,16)\black(17,16)\black(20,16)\black(1,17)\black(3,17)
  \black(4,17)\black(5,17)\black(6,17)\black(7,17)\black(8,17)
  \black(9,17)\black(10,17)\black(11,17)\black(12,17)\black(13,17)
  \black(17,17)\black(18,17)\black(20,17)\black(1,18)\black(14,18)
  \black(15,18)\black(18,18)\black(20,18)\black(1,19)\black(2,19)
  \black(3,19)\black(4,19)\black(5,19)\black(6,19)\black(7,19)
  \black(8,19)\black(9,19)\black(10,19)\black(11,19)\black(12,19)
  \black(13,19)\black(15,19)\black(16,19)\black(17,19)
  \black(20,19)\black(20,20)
\end{hexboard}
  \]
  \caption{Optimal winning paths for $n=1,\ldots,20$.}
  \label{fig:witnesses}
\end{figure}
% ......................................................................

% ......................................................................
\begin{figure}[p]
  \[
  \begin{array}{r|r|r|r}
    n & \mbox{length} & \mbox{bound} & \mbox{count} \\\hline
    1 & 1 & N/A & 1 \\
    2 & 2 & N/A & 3 \\
    3 & 5 & N/A & 1 \\
    4 & 8 & N/A & 4 \\
    5 & 11 & N/A & 23 \\\hline
    6 & 16 & 16 & 51 \\
    7 & 23 & 23 & 20 \\
    8 & 30 & 30 & 115 \\
    9 & \highlight{37} & 38 & 5568 \\
    10 & 47 & 47 & 12 \\\hline
    11 & 57 & 57 & 3521 \\
    12 & 69 & 69 & 40 \\
    13 & 81 & 81 & 1058 \\
    14 & 94 & 94 & 2104 \\
    15 & 109 & 109 & 668 \\\hline
    16 & 124 & 124 & 7540 \\
    17 & 140 & 140 & 1298 \\
    18 & 157 & 157 & 83648 \\
    19 & 175 & 175 & 16631833 \\
    20 & 195 & 195 & 70630 \\
  \end{array}
  \]  
  \caption{For each $n=1,\ldots,20$, the length of the optimal winning
    path (``length''), the bound on the length predicted by
    Theorem~\ref{thm:bounds} (``bound''), and the number of optimal
    paths (``count''). The first and last of these are sequences
  A375298 and A375299 in the On-Line Encyclopedia of Integer Sequences
  {\cite{OEIS}}.}
  \label{fig:path-lengths}
\end{figure}
% ......................................................................

% ......................................................................
\begin{figure}[p]
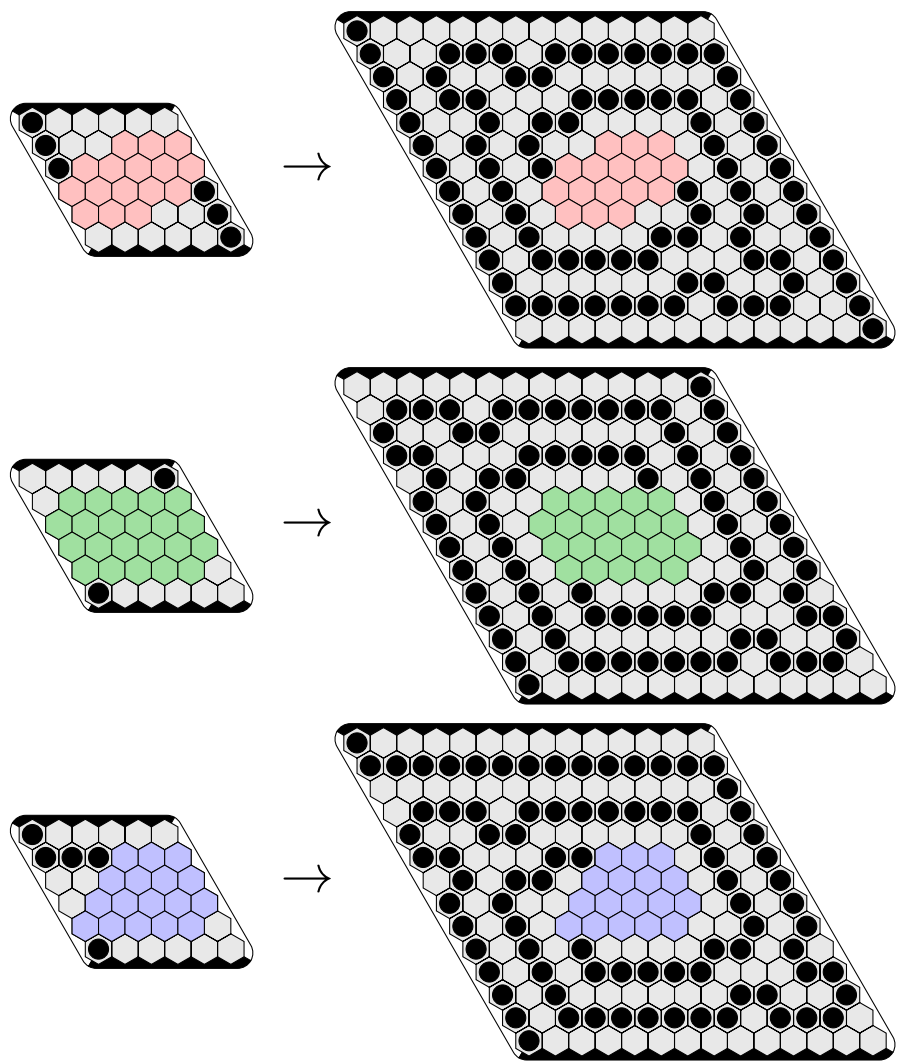

  \def\myarrow{\quad\scalebox{2}{$\to$}}
  \[
  \begin{hexboard}[scale=0.5, baseline={(current bounding box.center)}]
    \rotation{-30}
    \board(6,6)
    \black(1,1)\black(1,2)\black(1,3)
    \black(6,4)\black(6,5)\black(6,6)
    \cellcolor{cellpink}
    \foreach\i in {4,...,6} {\hex(\i,2)}
    \foreach\i in {2,...,6} {\hex(\i,3)}
    \foreach\i in {1,...,5} {\hex(\i,4)}
    \foreach\i in {1,...,3} {\hex(\i,5)}
  \end{hexboard}
  \myarrow
  \begin{hexboard}[scale=0.5, baseline={(current bounding box.center)}]
    \rotation{-30}
    \board(14,14)
    \cellcolor{cellpink}
    \foreach\i in {8,...,10} {\hex(\i,6)}
    \foreach\i in {6,...,10} {\hex(\i,7)}
    \foreach\i in {5,...,9} {\hex(\i,8)}
    \foreach\i in {5,...,7} {\hex(\i,9)}
    \black(1,1)\black(1,2)\black(1,3)\black(1,4)\black(1,5)
    \black(1,6)\black(1,7)\black(1,8)\black(1,9)\black(1,10)
    \black(1,11)\black(1,12)\black(1,13)\black(2,13)\black(3,3)
    \black(3,4)\black(3,6)\black(3,7)\black(3,8)\black(3,9)
    \black(3,10)\black(3,11)\black(3,13)\black(4,2)\black(4,4)
    \black(4,5)\black(4,11)\black(4,13)\black(5,2)\black(5,6)
    \black(5,7)\black(5,11)\black(5,13)\black(6,2)\black(6,3)
    \black(6,5)\black(6,11)\black(6,13)\black(7,3)\black(7,5)
    \black(7,11)\black(7,13)\black(8,2)\black(8,4)\black(8,10)
    \black(8,12)\black(9,2)\black(9,4)\black(9,10)\black(9,12)
    \black(9,13)\black(10,2)\black(10,4)\black(10,8)\black(10,9)
    \black(10,13)\black(11,2)\black(11,4)\black(11,10)
    \black(11,11)\black(11,13)\black(12,2)\black(12,4)\black(12,5)
    \black(12,6)\black(12,7)\black(12,8)\black(12,9)\black(12,11)
    \black(12,12)\black(13,2)\black(14,2)\black(14,3)\black(14,4)
    \black(14,5)\black(14,6)\black(14,7)\black(14,8)\black(14,9)
    \black(14,10)\black(14,11)\black(14,12)\black(14,13)
    \black(14,14)
  \end{hexboard}
  \]
  \[
  \begin{hexboard}[scale=0.5, baseline={(current bounding box.center)}]
    \rotation{-30}
    \board(6,6)
    \black(6,1)\black(1,6)
    \cellcolor{cellgreen}
    \foreach\i in {2,...,6} {\hex(\i,2)}
    \foreach\i in {1,...,6} {\hex(\i,3)}
    \foreach\i in {1,...,6} {\hex(\i,4)}
    \foreach\i in {1,...,5} {\hex(\i,5)}
  \end{hexboard}
  \myarrow
  \begin{hexboard}[scale=0.5, baseline={(current bounding box.center)}]
    \rotation{-30}
    \board(14,14)
    \cellcolor{cellgreen}
    \foreach\i in {6,...,10} {\hex(\i,6)}
    \foreach\i in {5,...,10} {\hex(\i,7)}
    \foreach\i in {5,...,10} {\hex(\i,8)}
    \foreach\i in {5,...,9} {\hex(\i,9)}
    \black(14,1)\black(2,2)\black(3,2)\black(4,2)\black(6,2)
    \black(7,2)\black(8,2)\black(9,2)\black(10,2)\black(11,2)
    \black(12,2)\black(14,2)\black(1,3)\black(4,3)\black(5,3)
    \black(12,3)\black(14,3)\black(1,4)\black(2,4)\black(6,4)
    \black(7,4)\black(8,4)\black(9,4)\black(10,4)\black(12,4)
    \black(14,4)\black(2,5)\black(4,5)\black(5,5)\black(10,5)
    \black(12,5)\black(14,5)\black(1,6)\black(3,6)\black(12,6)
    \black(14,6)\black(1,7)\black(3,7)\black(12,7)\black(14,7)
    \black(1,8)\black(3,8)\black(12,8)\black(14,8)\black(1,9)
    \black(3,9)\black(12,9)\black(14,9)\black(1,10)\black(3,10)
    \black(5,10)\black(10,10)\black(11,10)\black(13,10)
    \black(1,11)\black(3,11)\black(5,11)\black(6,11)\black(7,11)
    \black(8,11)\black(9,11)\black(13,11)\black(14,11)
    \black(1,12)\black(3,12)\black(10,12)\black(11,12)
    \black(14,12)\black(1,13)\black(3,13)\black(4,13)\black(5,13)
    \black(6,13)\black(7,13)\black(8,13)\black(9,13)\black(11,13)
    \black(12,13)\black(13,13)\black(1,14)
  \end{hexboard}
  \]
  \[
  \begin{hexboard}[scale=0.5, baseline={(current bounding box.center)}]
    \rotation{-30}
    \board(6,6)
    \black(1,1)\black(1,2)\black(2,2)\black(3,2)\black(1,6)
    \cellcolor{cellblue}
    \foreach\i in {4,...,6} {\hex(\i,2)}
    \foreach\i in {3,...,6} {\hex(\i,3)}
    \foreach\i in {2,...,6} {\hex(\i,4)}
    \foreach\i in {1,...,5} {\hex(\i,5)}
  \end{hexboard}
  \myarrow
  \begin{hexboard}[scale=0.5, baseline={(current bounding box.center)}]
    \rotation{-30}
    \board(14,14)
    \cellcolor{cellblue}
    \foreach\i in {8,...,10} {\hex(\i,6)}
    \foreach\i in {7,...,10} {\hex(\i,7)}
    \foreach\i in {6,...,10} {\hex(\i,8)}
    \foreach\i in {5,...,9} {\hex(\i,9)}
    \black(1,1)\black(1,2)\black(2,2)\black(3,2)\black(4,2)
    \black(5,2)\black(6,2)\black(7,2)\black(8,2)\black(9,2)
    \black(10,2)\black(11,2)\black(12,2)\black(13,2)\black(14,2)
    \black(14,3)\black(2,4)\black(3,4)\black(4,4)\black(6,4)
    \black(7,4)\black(8,4)\black(9,4)\black(10,4)\black(11,4)
    \black(12,4)\black(14,4)\black(1,5)\black(4,5)\black(5,5)
    \black(12,5)\black(14,5)\black(1,6)\black(2,6)\black(6,6)
    \black(7,6)\black(12,6)\black(14,6)\black(2,7)\black(4,7)
    \black(5,7)\black(12,7)\black(14,7)\black(1,8)\black(3,8)
    \black(12,8)\black(14,8)\black(1,9)\black(3,9)\black(12,9)
    \black(14,9)\black(1,10)\black(3,10)\black(5,10)\black(10,10)
    \black(11,10)\black(13,10)\black(1,11)\black(3,11)
    \black(5,11)\black(6,11)\black(7,11)\black(8,11)\black(9,11)
    \black(13,11)\black(14,11)\black(1,12)\black(3,12)
    \black(10,12)\black(11,12)\black(14,12)\black(1,13)
    \black(3,13)\black(4,13)\black(5,13)\black(6,13)\black(7,13)
    \black(8,13)\black(9,13)\black(11,13)\black(12,13)
    \black(13,13)\black(1,14)
  \end{hexboard}
  \]
  \caption{Recursive construction of winning paths for $n>20$. Given a
    winning path of length $k$ for board size $n$ of one of the types
    shown, construct a winning path of length $k+8n+30$ for board size
    $n+8$.}
  \label{fig:witnesses-extend}
\end{figure}
% ......................................................................
\clearpage %%% ###

% ----------------------------------------------------------------------
\section{Enumeration}

The optimal winning paths in Figure~\ref{fig:witnesses} and the path
counts in Figure~\ref{fig:path-lengths} were computed by exhaustive
search. The reader may perhaps wonder how such searches can be
performed. A first idea is to enumerate all possible paths in a
depth-first search, starting with paths of length 1 and then
repeatedly extending them by one stone in every possible way, as
indicated schematically in Figure~\ref{fig:enum-naive}(a).
% ......................................................................
\begin{figure}
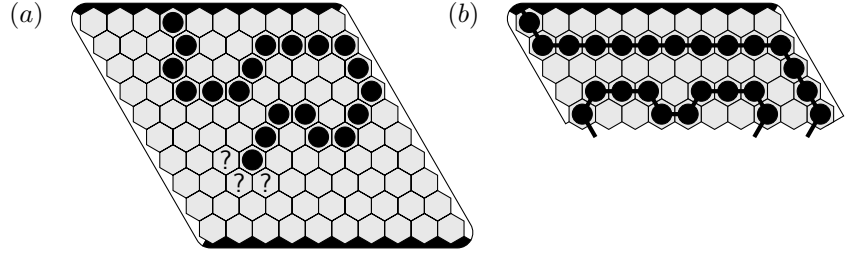

  \[
  (a)\quad
  \begin{hexboard}[scale=0.5, baseline={(0,0)}] % 10x10
    \rotation{-30}
    \board(10,10)
    \black(4,1)\black(4,2)\black(7,2)\black(8,2)\black(9,2)
    \black(10,2)\black(3,3)\black(6,3)\black(10,3)\black(3,4)
    \black(4,4)\black(5,4)\black(10,4)\black(6,5)\black(7,5)
    \black(9,5)\black(5,6)\black(7,6)\black(8,6)\black(4,7)
    \cell(3,7)\label{\lbl{?}}
    \cell(3,8)\label{\lbl{?}}
    \cell(4,8)\label{\lbl{?}}
  \end{hexboard}
  \hspace{-1em}
  (b)\quad
  \begin{hexboard}[scale=0.5, baseline={(0,0)}] % 10x10
    \rotation{-30}
    \foreach\i in {1,...,5} {
      \foreach\j in {1,...,10} {
      \hex(\j,\i)
      }
    }
    \edge(1,1)(10,1)
    \edge[\noobtusecorner](1,1)(1,4.5)
    \edge[\noacutecorner](10,1)(10,5)
    \black(1,1)\black(1,2)\black(2,2)\black(3,2)\black(4,2)
    \black(5,2)\black(6,2)\black(7,2)\black(8,2)\black(9,2)
    \black(10,2)\black(10,3)\black(2,4)\black(3,4)\black(4,4)
    \black(6,4)\black(7,4)\black(8,4)\black(10,4)\black(1,5)
    \black(4,5)\black(5,5)\black(8,5)\black(10,5)
    \draw[ultra thick] \coord(1,0.5) -- \coord(1,2) -- \coord(10,2) -- \coord(10,5) -- \coord(9,6);
    \draw[ultra thick] \coord(1,6) -- \coord(1,5) -- \coord(2,4) -- \coord(4,4) -- \coord(4,5) -- \coord(5,5) -- \coord(6,4) -- \coord(8,4) -- \coord(8,5) -- \coord(7,6);
  \end{hexboard}
  \]
  \caption{Enumeration of winning paths. (a) Naive method. (b)
    Row-wise method.}
  \label{fig:enum-naive}
\end{figure}
% ......................................................................
The problem with this naive enumeration method is that it is very
slow. The number of paths grows exponentially with the board size, and
even if we take some measures to avoid paths that have no chance of
becoming optimal (such as those that enclose a large number of useless
triangles), the runtime remains exponential. This method works for
sufficiently small boards, but is not feasible for sizes much beyond
$10\times 10$.

The key to efficient enumeration is to consider partial boards, rather
than partial paths. Consider the partial board in
Figure~\ref{fig:enum-naive}(b). It holds a number of path segments
that may (or may not) become part of a longest winning path. The lower
boundary of this region is formed by a row of 10 cells. Some of these
boundary cells are occupied by stones, and others are empty. Some of
the stones have a ``path stub'' coming out of them. These path stubs
have a direction, pointing left or right. Within the region, each of
the path stubs is either connected to a unique other path stub or to
the top board edge. We call all of this information (which cells along
the bottom row are occupied, the set of path stubs and their
directions, and the connectivity between the path stubs) the
\emph{boundary type} of the partial board.  The key observation is
that for a fixed board width (but arbitrary height), there are only
finitely many possible boundary types.  For each boundary type and
each height of partial boards, we only need to enumerate three pieces
of information: (1) the largest number of stones in any arrangement
for this boundary type and height, (2) how many arrangements exist
that use this largest number of stones, and (3) one example of such an
arrangement. Once we have collected this information for partial
boards of height $k$, it is easy to consider all possible ways of
adding one more row of stones, and infer the corresponding information
for partial boards of height $k+1$. Some care is required, for example,
to avoid creating disconnected cycles (this is why we must include
connectivity information in the boundary type). But with appropriate
attention to detail, this method allows us to enumerate all optimal
arrangements for all boundary types and all heights, provided that the
board width is fixed. Moreover, for any fixed board width, the runtime
scales linearly with the height.

An obvious limitation of the row-wise enumeration method is that the
runtime still grows exponentially as a function of the board
\emph{width}. However, the method is efficient enough to handle boards
of width up to 20 and arbitrary height. Moreover, there is an
analogous column-wise method that can be used for boards of height up
to 20 and arbitrary width.  The information in
Figures~\ref{fig:witnesses} and {\ref{fig:path-lengths}} was computed
using these methods.

% ----------------------------------------------------------------------
\section{Conclusions and future work}

We have determined the length of the longest winning path on Hex
boards of size $n\times n$, for all $n$. More generally, one may ask
the same question for non-rhombic boards, i.e., boards of size
$m\times n$. While it is possible that our methods can shed some light
on this more general question, we have left it for future work. Even
more generally, one can ask for longest induced paths in other types
of grids and graphs. But as noted in the introduction, the problem is
in general NP-hard, so one can only expect to solve it in special
cases.

% ----------------------------------------------------------------------
\section*{Acknowledgements}

This work was supported by the Natural Sciences and Engineering
Research Council of Canada (NSERC). I would like to thank the
anonymous referees for their thoughtful suggestions and
corrections. Special thanks to Eric Demer, who was the first to
compute all solutions up to $n=10$ and inspired me to do this work.

% ======================================================================
\bibliographystyle{abbrv}
\bibliography{hex-path}

\end{document}